\documentclass{amsart}

\usepackage{hyperref}
\usepackage{epsfig}
\usepackage{appendix}
\usepackage{amsfonts,amssymb,amsmath,amsthm,amscd}
\usepackage{latexsym}
\usepackage{mathtools}
\usepackage{graphicx}
\usepackage[all, cmtip]{xy}

\newtheorem{innercustomthm}{Theorem}
\newenvironment{thm}[1]
  {\renewcommand\theinnercustomthm{#1}\innercustomthm}
  {\endinnercustomthm}

\theoremstyle{plain}
		\newtheorem{theorem}{Theorem}[section]
		\newtheorem{lemma}[theorem]{Lemma}
		\newtheorem{corollary}[theorem]{Corollary}
		\newtheorem{proposition}[theorem]{Proposition}
		\newtheorem{claim}[theorem]{Claim}

		\newtheorem*{ntheorem}{Theorem}

		\newtheorem{fact}[theorem]{Fact}
\theoremstyle{definition}
		\newtheorem{definition}[theorem]{Definition}

\theoremstyle{remark}
		
		\newtheorem{example}[theorem]{Example}
				
\newcommand{\A}{{\mathbb{A}}}
\newcommand{\B}{{\mathbb{B}}}
\newcommand{\C}{{\mathbb{C}}}

\newcommand{\G}{{\mathbb{G}}}
\renewcommand{\H}{{\mathbb{H}}}

\newcommand{\Q}{{\mathbb{Q}}}
\newcommand{\R}{{\mathbb{R}}}

\newcommand{\Z}{{\mathbb{Z}}}

\newcommand{\Cb}{{\mathbf{C}}}

\newcommand{\Ocal}{{\mathcal{O}}}

\newcommand{\del}{\partial}

\newcommand{\id}{{\textup{id}}}

\DeclareMathOperator{\Diff}{Diff}

\DeclareMathOperator{\GL}{GL}

\DeclareMathOperator{\Hom}{Hom}

\DeclareMathOperator{\sheafhom}{\mathcal{H}\kern -.5pt \emph{om}}

\DeclareMathOperator{\Rep}{Rep}

\DeclareMathOperator{\SL}{SL}

\DeclareMathOperator{\SU}{SU}

\DeclareMathOperator{\Spec}{Spec}

\DeclareMathOperator{\tr}{tr}

\newcommand{\git}{\mathbin{
  \mathchoice{/\mkern-6mu/}% \displaystyle
    {/\mkern-6mu/}% \textstyle
    {/\mkern-5mu/}% \scriptstyle
    {/\mkern-5mu/}}}% \scriptscriptstyle

\begin{document}
\title{Arithmetic of curves on moduli of local systems}
\author{Junho Peter Whang}
\address{Dept of Mathematics, Massachusetts Institute of Technology, Cambridge MA}
\email{jwhang@mit.edu}
\date{\today}

\begin{abstract}
We investigate the arithmetic of algebraic curves on coarse moduli spaces for special linear rank two local systems on surfaces with fixed boundary traces. We prove a structure theorem for morphisms from the affine line into the moduli space. We show that the set of integral points on any nondegenerate algebraic curve on the moduli space can be effectively determined.
\end{abstract}

\maketitle

\setcounter{tocdepth}{1}
\tableofcontents

\section{Introduction} \label{sect:1}

\subsection{\unskip} \label{sect:1.1}
This is a continuation of our Diophantine study \cite{whang2} of moduli spaces for local systems on surfaces and their mapping class group dynamics. Let $\Sigma$ be a smooth compact oriented surface of genus $g$ with $n$ boundary curves satisfying $3g+n-3>0$. Let $X_k$ be the coarse moduli space of $\SL_2(\C)$-local systems on $\Sigma$ with prescribed boundary traces $k\in\A^n(\C)$. It is an irreducible complex affine algebraic variety of dimension $6g+2n-6$, and we showed in \cite{whang} that it is log Calabi-Yau if the surface has nonempty boundary. For $k\in\A^n(\Z)$, the variety $X_k$ admits a natural model over $\Z$. The mapping class group $\Gamma$ of the surface acts on $X_k$ via pullback of local systems, and an associated theory of descent on the integral points $X_k(\Z)$ was developed in \cite{whang2}. In this paper, we investigate the interplay between the dynamics of this action and the Diophantine geometry of algebraic curves on $X_k$.

\subsection{Main results} \label{sect:1.2}
We describe the contents of this paper. Relevant background on surfaces and their moduli of local systems is given in Section \ref{sect:2}, where we repeat material from \cite[Section 2]{whang2}. As in \cite{whang2}, let us say that a simple closed curve on $\Sigma$ is \emph{essential} if it cannot be continuously deformed into a point or a boundary curve on $\Sigma$. This paper is devoted to developing consequences of the following boundedness theorem \cite[Theorem 3]{whang2} for nonarchimedean systoles of local systems.

\begin{ntheorem}[\cite{whang2}]
Let $\Ocal$ be a discrete valuation ring with fraction field $F$. Given any representation $\rho:\pi_1\Sigma\to\SL_2(F)$ whose boundary traces all take values in $\Ocal$, there is an essential simple closed curve $a\subset\Sigma$ with $\tr\rho(a)\in\Ocal$.
\end{ntheorem}

In Section \ref{sect:3}, we apply the above theorem to the field of rational functions and prove our first main result, which is a structure theorem for morphisms from the affine line $\A^1$ into the moduli space $X_k$. Following \cite{whang2}, let us say that a possibly reducible algebraic variety $Z$ is \emph{parabolic} if it is covered by nonconstant morphisms $\A^1\to Z$. We also define a subvariety of $X_k$ to be \emph{degenerate} if it is contained in a parabolic subvariety of $X_k$, and \emph{nondegenerate} otherwise. The following theorem gives a modular characterization of the degenerate points of $X_k$.

\begin{theorem}
\label{theorem1}
A point $\rho\in X_k(\C)$ is degenerate if and only if
\begin{enumerate}
	\item \emph{(``parabolic curve'')} there is an essential simple closed curve $a\subset\Sigma$ such that $\tr\rho(a)=\pm2$, or
	\item \emph{(``parabolic pants'')} $(g,n,k)\neq(1,1,2)$ and there is a subsurface $\Sigma'\subset\Sigma$ of genus $0$ with $3$ boundary curves, each of which is an essential curve or a boundary curve of $\Sigma$, such that the restriction $\rho|\Sigma'$ is reducible.
\end{enumerate}
In particular, there is a parabolic proper closed subvariety $Z$ of $X_k$ such that every nonconstant morphism $\A^1\to X_k$ over $\C$ is mapping class group equivalent to one with image in $Z$.
\end{theorem}

Theorem \ref{theorem1} is reminiscent of a result of Sterk \cite{sterk} that the automorphism group of a projective K3 surface acts on the set of its smooth rational curve classes with finitely many orbits. It also has an interesting consequence (Corollary \ref{polycor}) that any polynomial deformation of a Fuchsian representation of surface group preserving the boundary traces must be isotrivial. Finally, Theorem \ref{theorem1} is used in formulating the main Diophantine result of \cite{whang2} for the integral points of $X_k$.

In Section \ref{sect:4}, we study the behavior of integral points on algebraic curves in $X_k$. For each curve $a\subset\Sigma$, let $\tr_a$ be the regular function on $X_k$ given by monodromy trace of local systems along $a$. We define an algebraic curve $C\subset X_k$ to be \emph{integrable} if there is a pants decomposition $P$ of $\Sigma$ (i.e.~a maximal union of pairwise disjoint and nonisotopic essential simple closed curves) such that ${\tr_a}$ is constant on $C$ for every curve $a\subset P$. Otherwise, $C$ is \emph{nonintegrable}. Given an algebraic curve $C\subset X_k$ and an arbitrary subset $A\subseteq\C$, let us denote
$$C(A)=\{\rho\in V(\C):\text{$\tr_a(\rho)\in A$ for every essential simple closed curve $a\subset\Sigma$}\}.$$
We prove the following result, by applying the boundedness of nonarchimedean systoles on local systems to function fields of algebraic curves.

\begin{theorem}
\label{theorem3}
If $C\subset X_k$ is a geometrically irreducible nonintegrable algebraic curve, then $C(A)$ is finite for any closed discrete set $A\subset\C$.
\end{theorem}

Moreover, our method will show that, given an embedding of $C$ into affine space, the sizes of the coordinates of $C(A)$ from the theorem can be effectively determined. One application is the following. For each positive squarefree integer $d$, let $O_d\subset\C$ denote the ring of integers of the imaginary quadratic field $\Q(\sqrt{-d})$. Applying Theorem \ref{theorem3} with $A=\bigcup_{d>0}O_d$, we conclude that a nonintegrable curve in $X_k$ has at most finitely many imaginary quadratic integral points. As a special case, this recovers the finiteness result of Long and Reid \cite{lr} for imaginary quadratic integral points on character curves of one-cusped hyperbolic three-manifolds; see Section \ref{sect:4} for details. Our approach to finiteness of integral points on nonintegrable curves shares its basis with the so-called Runge's method, described in \cite{zannier}.

By combining Theorem \ref{theorem3} with an analysis of integrable algebraic curves using Baker's theory on linear forms in logarithms, we also obtain the following result in Section \ref{sect:4}. Let us define an element in the mapping class group of $\Sigma$ to be a \emph{multitwist} if it is given by a product of commuting Dehn twists (and their powers) along essential curves in a pants decomposition of $\Sigma$. By a $1$-dimensional algebraic torus we shall mean an irreducible algebraic curve of genus $0$ with $2$ punctures.

\begin{theorem}
\label{theorem4}
Let $C\subset X_k$ be a geometrically irreducible nondegenerate algebraic curve over $\Z$. Then $C(\Z)$ can be effectively determined, and
\begin{enumerate}
	\item $C(\Z)$ is finite, or
	\item $C$ is the image of a $1$-dimensional algebraic torus preserved by a nontrivial multitwist, under which $C(\Z)$ consists of finitely many orbits.
\end{enumerate}
If moreover $C$ is not fixed pointwise by any nontrivial multitwist, the same result holds with $C(\Z)$ replaced by the set of all imaginary quadratic integral points on $C$.
\end{theorem}

Theorem \ref{theorem4} gives a complete analysis of integral points on every nondegenerately embedded curve $C\subset X_k$ with no intrinsic restrictions, e.g.~regarding the genus and number of punctures of the curve. The structure of Theorem \ref{theorem4} is strongly reminiscent of, and motivated by, classical Diophantine results on subvarieties of log Calabi-Yau varieties of linear type, such as algebraic tori and abelian varieties (cf.~Skolem's approach to the Thue equations \cite{bs}, as well as works of Vojta \cite{vojta2}, \cite{vojta3} and Faltings \cite{faltings}). Finally, for $(g,n)=(1,1)$ or $(0,4)$, the moduli space $X_k$ has an explicit presentation as an affine cubic algebraic surface with equation of the form
\begin{align*}
\tag{$*$}
x^2+y^2+z^2+xyz=ax+by+cz+d
\end{align*}
for some constants $a,b,c,d$ depending on $k$. Affine algebraic surfaces of this type were first studied by Markoff \cite{markoff}, where he introduced a form of nonlinear descent which essentially coincides the mapping class group action. Our work therefore specializes to the following result, which may be proved elementarily (but still using the group action).

\begin{corollary}
\label{corollary5}
On an affine algebraic surface with an equation of the form $(*)$, the integral solutions to any Diophantine equation over $\Z$ can be effectively determined.
\end{corollary}

The work of Ghosh and Sarnak \cite{gs} shows that, in the sense of proportions, almost all ``admissible'' Markoff type surfaces $X_k$ for $(g,n)=(1,1)$ have a Zariski dense set of integral points. Thus, Corollary \ref{corollary5} provides an infinite family of nontrivial ambient varieties of dimension two where every Diophantine equation over $\Z$ can be effectively solved.

\subsection{Ackowledgements} \label{sect:1.3}
This work was done as part of the author's Ph.D.~thesis at Princeton University. I thank my advisor Peter Sarnak and Phillip Griffiths for their guidance, encouragement, and insightful discussions. I also thank Rafael von K\"anel for useful historical remarks, and thank the referee for suggesting numerous improvements to the paper.

\section{Background} \label{sect:2}
This section collects relevant background on surfaces and their moduli of local systems, repeating material from \cite[Section 2]{whang2}. We also recall a boundedness result \cite[Theorem 1.3]{whang2} for nonarchimedean systoles of local systems in Section \ref{sect:2.4}, which will prove instrumental in our Diophantine analysis.

\subsection{Surfaces} \label{sect:2.1}
A \emph{surface} is an oriented two dimensional smooth manifold, which we assume to be compact with at most finitely many boundary components unless otherwise indicated. A connected surface is said to have type $(g,n)$ if it has genus $g$ and has $n$ boundary components. A \emph{curve} on a surface is an embedded copy of an unoriented circle, which we shall tacitly assume to be smooth in appropriate contexts. Given a surface $\Sigma$, we shall say that a curve $a\subset\Sigma$ is \emph{nondegenerate} if it does not bound a disk, and \emph{essential} if it is nondegenerate and is disjoint from, and not isotopic to, a boundary curve of $\Sigma$.

A \emph{multicurve} on $\Sigma$ is a finite union of disjoint curves on $\Sigma$. It is said to be \emph{nondegenerate}, resp.~\emph{essential}, if each of its components is. Given a surface $\Sigma$ and an essential multicurve $Q\subset\Sigma$, we denote by $\Sigma|Q$ the surface obtained by cutting $\Sigma$ along the curves in $Q$. A \emph{pants decomposition} of $\Sigma$ is an essential multicurve $P$ such that $\Sigma|P$ is a disjoint union of surfaces of type $(0,3)$. Equivalently, a pants decomposition is a maximal (with respect to inclusion) essential multicurve whose components are pairwise nonisotopic. If $\Sigma$ is a surface of type $(g,n)$ with $3g+n-3>0$, then any pants decomposition of $\Sigma$ consists of $3g+n-3$ essential curves.  An essential curve $a\subset\Sigma$ is \emph{separating} if the two boundary curves of $\Sigma|a$ corresponding to $a$ are on different connected components, and \emph{nonseparating} otherwise.

\subsubsection{Optimal generators} \label{sect:2.1.1}
Let $\Sigma$ be a surface of type $(g,n)$, and choose a base point $x\in\Sigma$. We have the \emph{standard presentation} of the fundamental group
\begin{equation}
\pi_1(\Sigma,x)=\langle \alpha_1,\beta_1',\cdots,\alpha_g,\beta_g',\gamma_1,\cdots,\gamma_n|[\alpha_1,\beta_1']\cdots[\alpha_g,\beta_g']\gamma_1\cdots \gamma_n\rangle
\label{eq:stdprs}
\end{equation}
where in particular $\gamma_1,\cdots,\gamma_n$ correspond to loops around the boundary curves of $\Sigma$. For $i=1,\cdots,g$, let $\beta_i$ be the based loop traversing $\beta_i'$ in the opposite direction. We can choose the sequence of generating loops $(\alpha_1,\beta_1,\cdots,\alpha_g,\beta_g,\gamma_1,\cdots,\gamma_n)$ so that it satisfies the following:
\begin{enumerate}
	\item[(1)]each loop in the sequence is simple,	
	\item[(2)]any two distinct loops in the sequence intersect exactly once (at $x$), and
	\item[(3)]every product of distinct elements in the sequence preserving the cyclic ordering can be represented by a simple loop in $\Sigma$.
\end{enumerate}
Some examples of products alluded to in (3) are $\alpha_1\beta_g$, $\alpha_1\alpha_2\beta_2\beta_g$, and $\beta_g\gamma_n\alpha_1$. We refer to $(\alpha_1,\beta_1,\cdots,\alpha_g,\beta_g,\gamma_1,\cdots,\gamma_n)$ as an \emph{optimal sequence of generators} for $\pi_1\Sigma$. See Figure \ref{fig1} for an illustration of optimal generators for $(g,n)=(2,1)$.

\begin{figure}[ht]
    \centering
    \includegraphics{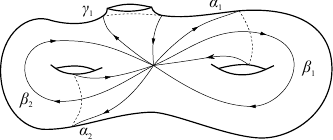}
    \caption{Optimal generators for $(g,n)=(2,1)$}
    \label{fig1}
\end{figure}

\subsubsection{Mapping class group} \label{sect:2.1.2}
Given a surface $\Sigma$, let $\Gamma=\Gamma(\Sigma)=\pi_0\Diff^+(\Sigma,\del\Sigma)$ denote its mapping class group. By definition, it is the group of isotopy classes of orientation preserving diffeomorphisms of $\Sigma$ fixing the boundary of $\Sigma$ pointwise. Given a (simple closed) curve $a\subset\Sigma$ disjoint from $\del\Sigma$, the associated (left) \emph{Dehn twist} $\tau_a\in\Gamma$ on $\Sigma$ is defined as follows. Let $S^1=\{z\in\C:|z|=1\}$ be the unit circle. Let $\tau$
be the diffeomorphism from $S^1\times[0,1]$ to itself given by $(z,t)\mapsto (ze^{2\pi i\xi(t)},t)$ where $\xi(t)$ is a smooth bump function of $t\in[0,1]$ that is $0$ on a neighborhood of $0$ and $1$ on a neighborhood of $1$.
Choose a closed tubular neighborhood $N$ of $a$ in $\Sigma$, and an orientation preserving diffeomorphism $f:N\to S^1\times[0,1]$. The Dehn twist $\tau_a$ is given by
$$\tau_a(x)=\left\{\begin{array}{l l} f^{-1}\circ\tau\circ f(x)&\text{if $x\in N$},\\x &\text{otherwise.}\end{array}\right.$$
The class of $\tau_a$ in $\Gamma$ is independent of the choices involved above, and depends only on the isotopy class of $a$. It is a standard fact that $\Gamma=\Gamma(\Sigma)$ is generated by Dehn twists along simple closed curves in $\Sigma$ (see \cite[Chapter 4]{fm}).

\subsection{Character varieties} \label{sect:2.2}
Throughout this paper, an \emph{algebraic variety} is a scheme of finite type over a field. Given an affine variety $V$ over a given field $k$, we denote by $k[V]$ its coordinate ring over $k$. If moreover $V$ is integral, then $k(V)$ denotes its function field over $k$. Given a commutative ring $A$ with unity, the elements of $A$ will be referred to as \emph{regular functions} on the affine scheme $\Spec A$.

\subsubsection{Character varieties of groups} \label{sect:2.2.1}

Let $\pi$ be a finitely generated group. Its ($\SL_2$) \emph{representation variety} $\Rep(\pi)$ is the affine scheme determined by the functor
$$A\mapsto\Hom(\pi,\SL_2(A))$$
for every commutative ring $A$. Given a sequence of generators of $\pi$ with $m$ elements, we have a presentation of $\Rep(\pi)$ as a closed subscheme of $\SL_2^m$ defined by equations coming from relations among the generators. For each $a\in\pi$, let $\tr_a$ be the regular function on $\Rep(\pi)$ given by $\rho\mapsto\tr\rho(a)$.

The ($\SL_2$) \emph{character variety} of $\pi$ over $\C$ is the affine invariant theoretic quotient
$$X(\pi)=\Rep(\pi)\git\SL_2=\Spec\C[\Rep(\pi)]^{\SL_2(\C)}$$
under the simultaneous conjugation action of $\SL_2$. Note that the regular function $\tr_{a}$ for each $a\in\pi$ descends to a regular function on $X(\pi)$. Moreover, $X(\pi)$ has a natural model over $\Z$, defined as the spectrum of
$$R(\pi)=\Z[\tr_a:a\in\pi]/(\tr_{1}-2,\tr_a\tr_b-\tr_{ab}-\tr_{ab^{-1}}).$$
The relations in the above presentation arise from the fact that the trace of the $2\times 2$ identity matrix is $2$, and $\tr(A)\tr(B)=\tr(AB)+\tr(AB^{-1})$ for every $A,B\in\SL_2(\C)$.

Given an integral domain $A$ with fraction field $F$ of characteristic zero, the $A$-points of $X(\pi)$ parametrize the Jordan equivalence classes of $\SL_2(\overline F)$-representations of $\pi$ having character valued in $A$ (see \cite[Proposition 6.1]{simpson}). Here, following \cite{simpson}, we say that two finite-dimensional linear representations of $\pi$ are Jordan equivalent if they admit composition series with isomorphic graded representations. Since a semisimple finite-dimensional representation of a group over a field of characteristic zero is determined by its character \cite[Chapter XVII, Section 3, Corollary 3.8]{lang}, we see in particular two representations $\rho:\pi\to\SL_2(\C)$ are Jordan equivalent if and only if they have the same character. (It is not true in general that, for a reductive algebraic group $G\leq\GL_r$ over $\C$, the points of $\Hom(\pi,G)\git G$ are determined by their characters.) We refer to \cite{horowitz}, \cite{ps}, \cite{saito} for further details on $\SL_2$-character varieties.

\begin{example}
\label{exfree}
We refer to Goldman \cite{goldman2} for details of examples below. Let $F_m$ denote the free group on $m\geq1$ generators $a_1,\cdots,a_m$.
\begin{enumerate}
	\item[(1)] We have $\tr_{a_1}:X(F_1)\simeq\A^1$.
	\item[(2)] We have $(\tr_{a_1},\tr_{a_2},\tr_{a_1a_2}):X(F_2)\simeq\A^3$ by Fricke \cite[Section 2.2]{goldman2}.
	\item[(3)] The coordinate ring $\Q[X(F_3)]$ is the quotient of the polynomial ring
	$$\Q[\tr_{a_1},\tr_{a_2},\tr_{a_3},\tr_{a_1a_2},\tr_{a_2a_3},\tr_{a_1a_3},\tr_{a_1a_2a_3},\tr_{a_1a_3a_2}]$$
	by the ideal generated by two elements
$$
\tr_{a_1a_2a_3}+\tr_{a_1a_3a_2}-(\tr_{a_1a_2}\tr_{a_3}+\tr_{a_1a_3}\tr_{a_2}+\tr_{a_2a_3}\tr_{a_1}-\tr_{a_1}\tr_{a_2}\tr_{a_3})
$$
and
\begin{align*}
\tr_{a_1a_2a_3}\tr_{a_1a_3a_2}&-\{(\tr_{a_1}^2+\tr_{a_2}^2+\tr_{a_3}^2)+(\tr_{a_1a_2}^2+\tr_{a_2a_3}^2+\tr_{a_1a_3}^2)&\\
&\quad -(\tr_{a_1}\tr_{a_2}\tr_{a_1a_2}+\tr_{a_2}\tr_{a_3}\tr_{a_2a_3}+\tr_{a_1}\tr_{a_3}\tr_{a_1a_3})\\
&\quad +\tr_{a_1a_2}\tr_{a_2a_3}\tr_{a_1a_3}-4\}.
\end{align*}
\end{enumerate}
\end{example}

We record the following, which is attributed by Goldman \cite{goldman2} to Vogt \cite{vogt}.

\begin{lemma}
\label{rellem}
Given a finitely generated group $\pi$ and $a_1,a_2,a_3,a_4\in \pi$, we have
\begin{align*}
2{\tr_{a_1a_2a_3a_4}}&={\tr_{a_1}}{\tr_{a_2}}{\tr_{a_3}}{\tr_{a_4}}+{\tr_{a_1}}{\tr_{a_2a_3a_4}}+{\tr_{a_2}}{\tr_{a_3a_4a_1}}+{\tr_{a_3}}{\tr_{a_4a_1a_2}}\\
&\quad +{\tr_{a_4}}{\tr_{a_1a_2a_3}}+{\tr_{a_1a_2}}{\tr_{a_3a_4}}+{\tr_{a_4a_1}}{\tr_{a_2a_3}}-{\tr_{a_1a_3}}{\tr_{a_2a_4}}\\
&\quad -{\tr_{a_1}}{\tr_{a_2}}{\tr_{a_3a_4}}-{\tr_{a_3}}{\tr_{a_4}}{\tr_{a_1a_2}}-{\tr_{a_4}}{\tr_{a_1}}{\tr_{a_2a_3}}-{\tr_{a_2}}{\tr_{a_3}}{\tr_{a_4a_1}}.
\end{align*}
\end{lemma}

The above computation implies the following fact, which forms a special case of Procesi's theorem \cite{procesi} that ring of invariants of tuples of $N\times N$ matrices under simultaneous conjugation are (finitely) generated by the trace functions of products of matrices.

\begin{fact}
\label{fact}
If $\pi$ is a group generated by $a_1,\cdots,a_m$, then $\Q[X(\pi)]$ is generated as a $\Q$-algebra by the collection $\{\tr_{a_{i_1}\cdots a_{i_k}}:1\leq i_1<\cdots<i_k\leq m\}_{1\leq k\leq 3}$.
\end{fact}

\subsubsection{Moduli of local systems on manifolds} \label{sect:2.2.2}

Given a connected smooth (compact) manifold $M$, the coarse moduli space of local systems on $M$ that we shall study is the character variety $X(M)=X(\pi_1 M)$ of its fundamental group. The complex points of $X(M)$ parametrize the Jordan equivalence classes of $\SL_2(\C)$-local systems on $M$. More generally, given a smooth manifold $M=M_1\sqcup\cdots\sqcup M_m$ with finitely many connected components $M_i$, we define
$$X(M)=X(M_1)\times\cdots\times X(M_m).$$
The construction of the moduli space $X(M)$ is functorial in the manifold $M$. Any smooth map $f:M\to N$ of manifolds induces a morphism $f^*:X(N)\to X(M)$, depending only on the homotopy class of $f$, given by pullback of local systems.

Let $\Sigma$ be a surface. For each curve $a\subset\Sigma$, there is a well-defined regular function $\tr_a:X(\Sigma)\to X(a)\simeq\A^1$, which agrees with $\tr_{\alpha}$ for any $\alpha\in\pi_1\Sigma$ represented by a path freely homotopic to a parametrization of $a$. Implicit here is the observation that $\tr_\alpha$ is independent of the choice of an orientation for $a$ since $\tr(A)=\tr(A^{-1})$ for any matrix $A$ in $\SL_2$. The boundary curves $\del\Sigma$ of $\Sigma$ induce a natural morphism
$$\tr_{\del\Sigma}:X(\Sigma)\to X(\del\Sigma)\simeq\A^n$$
where the latter isomorphism is given by a choice of ordering $\del\Sigma=c_1\sqcup\cdots\sqcup c_n$ of the boundary curves $c_i$. The fibers of $\tr_{\del\Sigma}$
for $k\in\A^n$ will be denoted $X_k=X_k(\Sigma)$. Each $X_k$ is often referred to as a \emph{relative character variety} in the literature. If $\Sigma$ is a surface of type $(g,n)$ satisfying $3g+n-3>0$, the relative character variety $X_k(\Sigma)$ is an irreducible algebraic variety of dimension $6g+2n-6$.

We shall often simplify our notation by combining parentheses where applicable, e.g.~$X_k(\Sigma,\Z)=X_k(\Sigma)(\Z)$. Given a fixed surface $\Sigma$, a subset $K\subseteq X(\del\Sigma,\C)$, and a subset $A\subseteq\C$, we shall denote by
$$X_K(A)=X_K(\Sigma,A)$$
the set of all $\rho\in X(\Sigma,\C)$ such that $\tr_{\del\Sigma}(\rho)\in K$ and $\tr_a(\rho)\in A$ for every essential curve $a\subset\Sigma$. The following lemma shows that there is no risk of ambiguity with this notation.

\begin{lemma}
If $A$ is a subring of $\C$ and $k\in A^n$, then $X_k$ has a model over $A$ and $X_k(A)$ recovers the set of $A$-valued points of $X_k$ in the sense of algebraic geometry.
\end{lemma}

\begin{proof}
Let $A$ and $k\in A^n$ be as above. We have a model of $X_k$ over $A$ with coordinate ring $\Spec R(\pi_1\Sigma)\otimes_\Z A$. It is clear that an $A$-valued point in the sense of algebraic geometry corresponds to a point in $X_k(A)$. The converse follows from the observation, using the identity $\tr_a\tr_b=\tr_{ab}+\tr_{ab^{-1}}$, that $\tr_b$ for every $b\in\pi_1\Sigma$ can be written as a $\Z$-linear combination of products of traces $\tr_a$ for nondegenerate curves $a\subset\Sigma$.
\end{proof}

Similarly, given $k\in X(\del\Sigma,\C)$, a subvariety $V\subseteq X_k(\Sigma)$, and a subset $A\subseteq\C$, we shall denote
$$V(A)=\{\rho\in V(\C):\text{$\tr\rho(a)\in A$ for every essential curve $a\subset\Sigma$}\}.$$

Given an immersion $\Sigma'\to\Sigma$ of surfaces, we have the associated restriction
$$(-)|_{\Sigma'}:X(\Sigma)\to X(\Sigma').$$
The mapping class group $\Gamma(\Sigma)$ acts naturally on $X(\Sigma)$ by pullback of local systems, preserving the integral structure as well as each relative character variety $X_k(\Sigma)$ and the sets $X_K(\Sigma,A)$ defined above. The dynamical aspects of this action on the complex points of $X(\Sigma)$ are not fully understood, but have been studied on certain special subloci. These include the locus of $\SU(2)$-local systems (see \cite{goldman}) on $X$, and the Teichm\"uller locus parametrizing \emph{Fuchsian representations} associated to marked hyperbolic structures on $\Sigma$ with geodesic boundary. This paper is largely concerned with the descent properties of the dynamics on $X(\C)$ beyond the classical setting.

\subsubsection{Reconstruction} \label{sect:2.2.3}
Let $\Sigma$ be a surface of type $(g,n)$ with $3g+n-3>0$, and let $a\subset\Sigma$ be an essential curve. Let $x\in\Sigma$ be a base point lying on $a$, and let $\alpha$ be a simple based loop parametrizing $a$. We shall summarize the reconstruction of a representation $\rho:\pi_1(\Sigma,x)\to\SL_2(\C)$ from representations on connected components of $\Sigma|a$, as well as associated lifts of Dehn twists. Our main reference is Goldman-Xia \cite{gx}. There are two cases to consider, according to whether $a$ is separating or nonseparating.

{\bf Nonseparating curves.}
Suppose that $a$ is nonseparating, so $\Sigma|a$ is connected. Let $a_1$ and $a_2$ be the boundary curves of $\Sigma|a$ corresponding to $a$, and let $(x_i,\alpha_i)$ be the lifts of $(x,\alpha)$ to each $a_i$. We shall assume that we have chosen the numberings so that the interior of $\Sigma|a$ lies to the left as one travels along $\alpha_1$. Let $\beta$ be a simple loop on $\Sigma$ based at $x$, intersecting the curve $a$ once transversely at the base point, such that $\beta$ lifts to a path $\beta'$ in $\Sigma|a$ from $x_2$ to $x_1$. Let us denote by $\alpha_2'$ the loop based at $x_1$ given by the path $\alpha_2'=(\beta')^{-1}\alpha_2\beta'$,
where $(\beta')^{-1}$ refers to the path $\beta'$ traversed in the opposite direction. The immersion $\Sigma|a\to\Sigma$ induces an embedding $\pi_1(\Sigma|a,x_1)\to\pi_1(\Sigma,x)$, giving us the isomorphism
$$\pi_1(\Sigma,x)=(\pi_1(\Sigma|a,x_1)\vee\langle\beta\rangle)/(\alpha_2'=\beta^{-1}\alpha_1\beta).$$
Thus, any representation $\rho:\pi_1(\Sigma,x)\to\SL_2(\C)$ is determined uniquely by a pair $(\rho',B)$, where $\rho':\pi_1(\Sigma|a,x_1)\to\SL_2(\C)$ is a representation and $B\in\SL_2(\C)$ is an element such that $\rho'(\alpha_2')=B^{-1}\rho'(\alpha_1)B$, with the correspondence
$$\rho\mapsto(\rho',B)=(\rho|_{\pi_1(\Sigma|a,x_1)},\rho(\beta)).$$
We define an automorphism $\tau_\alpha$ of $\Hom(\pi_1(\Sigma,x),\SL_2)$ as follows. Given $\rho=(\rho_a,B)$, we set $\tau_\alpha(\rho_a,B)=(\rho_a,B')$ where $B'=\rho(\alpha)B$. This descends to the action $\tau_a$ of the left Dehn twist action along $a$ on the moduli space $X(\Sigma)$.

{\bf Separating curves.}
Suppose that $a$ is separating, so we have $\Sigma|a=\Sigma_1\sqcup\Sigma_2$ with each $\Sigma_i$ of type $(g_i,n_i)$ satisfying $2g_i+n_i-2>0$. Let $a_i$ be the boundary curve of $\Sigma_i$ corresponding to $a$. Let $(x_i,\alpha_i)$ be the lift of $(x,\alpha)$ to $a_i$. We shall assume that we have chosen the numberings so that the interior of $\Sigma_1$ lies to the left as one travels along $\alpha_1$. The immersions $\Sigma_i\hookrightarrow\Sigma$ of the surfaces induce embeddings
$\pi_1(\Sigma_i,x_i)\to\pi_1(\Sigma,x)$
of fundamental groups, and we have an isomorphism
$$\pi_1(\Sigma,x)\simeq(\pi_1(\Sigma_1,x_1)\vee\pi_1(\Sigma_2,x_2))/(\alpha_1=\alpha_2).$$
Thus, any representation
$\rho:\pi_1(\Sigma,x)\to\SL_2(\C)$
is determined uniquely by a pair $(\rho_1,\rho_2)$ of representations $\rho_i:\pi_1(\Sigma_i,x_i)\to\SL_2(\C)$ such that $\rho_1(\alpha_1)=\rho_2(\alpha_2)$, with the correspondence
$$\rho\mapsto(\rho_1,\rho_2)=(\rho|_{\pi_1(\Sigma_1,x_1)},\rho|_{\pi_1(\Sigma_2,x_2)}).$$
We define an automorphism $\tau_{\alpha}$ of $\Hom(\pi_1(\Sigma,x),\SL_2)$ as follows. For a representation $\rho=(\rho_1,\rho_2)$, we set $\tau_\alpha(\rho_1,\rho_2)=(\rho_1,\rho_2')$
where $$\rho_2'(\gamma)=\rho(\alpha)\rho_2(\gamma)\rho(\alpha)^{-1}$$
for every $\gamma\in\pi_1(\Sigma_2,x_2)$. This descends to the action $\tau_a$ of the left Dehn twist along $a$ on the moduli space $X(\Sigma)$.

\subsection{Markoff type cubic surfaces} \label{sect:2.3}
Here, we give a description of the moduli spaces $X_k(\Sigma)$ and their mapping class group dynamics for $(g,n)=(1,1)$ and $(0,4)$. These cases are distinguished by the fact that each $X_k$ is an affine cubic algebraic surface with an explicit equation.

\subsubsection{Case $(g,n)=(1,1)$} \label{sect:2.3.1}
Let $\Sigma$ be a surface of type $(g,n)=(1,1)$, i.e.~a one holed torus. Let $(\alpha,\beta,\gamma)$ be an optimal sequence of generators for $\pi_1\Sigma$. By Example \ref{exfree}.(2), we have an identification
$(\tr_\alpha,\tr_\beta,\tr_{\alpha\beta}):X(\Sigma)\simeq\A^3$.
From the trace relations in Section \ref{sect:2.2}, we obtain
\begin{align*}
\tr_\gamma&=\tr_{\alpha\beta\alpha^{-1}\beta^{-1}}=\tr_{\alpha\beta\alpha^{-1}}\tr_{\beta^{-1}}-\tr_{\alpha\beta\alpha^{-1}\beta}\\
&=\tr_{\beta}^2-\tr_{\alpha\beta}\tr_{\alpha^{-1}\beta}+\tr_{\alpha\alpha}=\tr_{\beta}^2-\tr_{\alpha\beta}(\tr_{\alpha^{-1}}\tr_{\beta}-\tr_{\alpha\beta})+\tr_{\alpha}^2-\tr_{1}\\
&=\tr_{\alpha}^2+\tr_{\beta}^2+\tr_{\alpha\beta}^2-\tr_{\alpha}\tr_{\beta}\tr_{\alpha\beta}-2.
\end{align*}
Writing $(x,y,z)=(\tr_{\alpha},\tr_{\beta},\tr_{\alpha\beta})$ so that each of the variables $x$, $y$, and $z$ corresponds to an essential curve on $\Sigma$ as depicted in Figure \ref{fig2}, the moduli space $X_k\subset X$ has an explicit presentation as an affine cubic algebraic surface in $\A_{x,y,z}^3$ with equation
$$x^2+y^2+z^2-xyz-2=k.$$
\begin{figure}[ht]
    \includegraphics{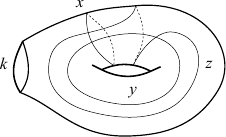}
    \caption{Curves on a surface of type $(1,1)$ with corresponding functions}
    \label{fig2}
\end{figure}

\subsubsection{Case $(g,n)=(0,4)$} \label{sect:2.3.2}
Let $\Sigma$ be a surface of type $(0,4)$, i.e.~a four holed sphere. Let $(\gamma_1,\cdots,\gamma_4)$ be an optimal sequence of generators for $\pi_1\Sigma$. Let us set
$$(x,y,z)=(\tr_{\gamma_1\gamma_2},\tr_{\gamma_2\gamma_3},\tr_{\gamma_1\gamma_3}),$$
so that each of the variables corresponds to an essential curve on $\Sigma$ as depicted in Figure \ref{fig3}. By Example \ref{exfree}.(3), for $k=(k_1,k_2,k_3,k_4)\in\A^4(
\C)$ the relative character variety $X_k=X_k(\Sigma)$ is an affine cubic algebraic surface in $\A_{x,y,z}^3$ given by the equation
$$x^2+y^2+z^2+xyz=Ax+By+Cz+D$$
with
$$\left\{\begin{array}{l}A=k_1k_2+k_3k_4\\B=k_1k_4+k_2k_3\\C=k_1k_3+k_2k_4\end{array}\right.\quad\text{and}\quad D=4-\sum_{i=1}^4k_i^2-\prod_{i=1}^4k_i.$$

\begin{figure}[ht]
    \centering
    \includegraphics{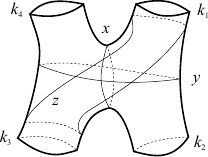}
    \caption{Curves on surfaces of type $(0,4)$ with corresponding functions}
    \label{fig3}
\end{figure}

\subsection{Nonarchimedean systoles} \label{sect:2.4}
In \cite{whang2}, we proved the following result.

\begin{theorem}
\label{nas}
Let $\Ocal$ be a discrete valuation ring with fraction field $F$. Given any representation $\rho:\pi_1\Sigma\to\SL_2(F)$ whose boundary traces all take values in $\Ocal$, there is an essential curve $a\subset\Sigma$ with $\tr\rho(a)\in\Ocal$.
\end{theorem}

\begin{corollary}
\label{corthm1}
Let $\Ocal$ and $F$ be as above. If $F$ has characteristic zero, then for any $\rho\in X_k(F)$ with $k\in\Ocal^n$ there is an essential curve $a\subset\Sigma$ such that $\tr\rho(a)\in\Ocal$.
\end{corollary}

\begin{proof}
Given $\rho\in X_k(F)$, since $F$ has characteristic zero there exists a finite field extension $F'/F$ such that $\rho$ is the class of a representation $\rho':\pi_1\Sigma\to\SL_2(F')$. Choosing an extension of the valuation on $F$ to $F'$, let $\Ocal'\subset F'$ be the associated valuation ring. It follows from Theorem \ref{theorem1} and our hypothesis $k\in\Ocal^n\subseteq(\Ocal')^n$ that there is an essential curve $a\subset\Sigma$ with $\tr\rho'(a)\in\Ocal'$. This then implies that $\tr\rho(a)=\tr\rho'(a)\in\Ocal'\cap F=\Ocal$, which is the desired result.
\end{proof}

Below, we record a special case of Corollary \ref{corthm1}, which plays a crucial role when we analyze the structure of morphisms from the affine line to $X_k$ in Section \ref{sect:3}.

\begin{lemma}
\label{linelem}
Given any morphism $f:\A^1\to X_k$ over $\C$, there is an essential curve $a\subset\Sigma$ such that $\tr_a\circ f:\A^1\to\A^1$ is constant.
\end{lemma}

\begin{proof}
A morphism $f:\A^1\to X_k$ corresponds to a $\C[t]$-valued point of $X_k$, giving rise to a $\C(t)$-valued point $\rho_f\in X_k(\C(t))$. Applying Corollary \ref{corthm1} with $F=\C(t)$, with discrete valuation given by the order of vanishing at $\infty$, we deduce that there is an essential curve $a\subset\Sigma$ such that $\tr\rho_f(a)={\tr_a}\circ f:\A^1\to\A^1$
has no pole at $\infty$, which implies that it must be constant, as desired.
\end{proof}

\section{Parabolic subvarieties} \label{sect:3}
Let $\Sigma$ be a surface of type $(g,n)$ satisfying $3g+n-3>0$, and let $X_k=X_k(\Sigma)$ for $k\in X(\del\Sigma,\C)$ be a relative character variety of $\Sigma$. Given a pants decomposition $P$ of $\Sigma$, the immersion $P\to\Sigma$ induces a morphism
$${\tr_P}:X_k\to X(P)\simeq\A^{3g+n-3}$$
whose fibers for $t\in X(P,\C)$ will be denoted $X_{k,t}^P=\tr_P^{-1}(t)$. Let
$$(-)|_{\Sigma|P}:X_k\to X(\Sigma|P)=X(\Sigma_1)\times\cdots\times X(\Sigma_{2g+n-2})$$
be the morphism induced by the immersion $\Sigma|P\to\Sigma$, where the product on the right hand side is taken over the connected components $\Sigma_i$ of $\Sigma|P$. Since $\pi_1\Sigma$ is a free group of rank $2$ and a point of $X(\Sigma_i)\simeq\A^3$ (see Example \ref{exfree}.(2)) is determined by the value of its traces along the boundary curves of $\Sigma_i$, it follows that $(-)|_{\Sigma|P}$ is constant along each fiber $X_{k,t}^P$. We make the following definition.

\begin{definition}
\label{perfdef}
Let $(P,t)$ be as above. The fiber $X_{k,t}^P$ is \emph{perfect} if
\begin{enumerate}
	\item[\textup{(1)}] $\tr_a(X_{k,t}^P)\neq\pm2$ for every curve $a\subseteq P$, and
	\item[\textup{(2)}] each factor of $(X_{k,t}^P)|_{\Sigma|P}$ is irreducible, or $(g,n,k)=(1,1,2)$.
\end{enumerate}
\end{definition}

The first part of condition (2) in the definition above means that, for each connected component $\Sigma_i$ of $\Sigma|P$, the point $(X_{k,t}^P)|_{\Sigma_i}$ is represented by an irreducible local system on $\Sigma_i$, or an irreducible representation $\pi_1\Sigma_i\to\SL_2(\C)$.

The above definition is motivated by the following theorem, which is the main result of this section. Recall from Section \ref{sect:1} that an algebraic variety $Z$ over $\C$ is said to be \emph{parabolic} if every closed point of $Z$ lies in the image of some nonconstant morphism $\A^1\to Z$. In the following, a pair $(P,t)$ will denote a pants decomposition $P$ of $\Sigma$ and an element $t\in X(P,\C)$.

\begin{theorem}
\label{linethm}
We have the following.
\begin{enumerate}
	\item[\textup{A.}] For each $(P,t)$, the fiber $X_{k,t}^P$ is either perfect or parabolic.
	\item[\textup{B.}] For any nonconstant morphism $f:\A^1\to X_k$, there is a parabolic fiber $X_{k,t}^P$ for some $(P,t)$ containing the image of $f$.
\end{enumerate}
\end{theorem}

The remainder of this section is organized as follows. In Section \ref{sect:3.1}, we give a proof of Theorem \ref{linethm} in the cases $(g,n)=(1,1)$ and $(0,4)$. The moduli space $X_k$ in these cases is an explicitly defined algebraic surface, making the proof easier. We prove the general case of Theorem \ref{linethm} in Section \ref{sect:3.2}. In Section \ref{sect:3.3}, we describe the consequences of Theorem \ref{linethm}, including Theorem \ref{theorem1} as well as a rigidity result (Corollary \ref{polycor}) for certain polynomial deformations of Fuchsian representations of surface groups.

\subsection{Base cases} \label{sect:3.1}
In this subsection, we give a proof of Theorem \ref{linethm} in the cases $(g,n)=(1,1)$ and $(0,4)$. We shall refer the reader to Section \ref{sect:2.3} for explicit presentations of $X_k$ in these cases. First, it is useful to record an elementary lemma.

\begin{lemma}
\label{redlem}
If $\Sigma$ is a surface of type $(0,3)$, then $k=(k_1,k_2,k_3)\in X(\Sigma)\simeq\A^3$ is reducible if and only if $k_1^2+k_2^2+k_3^2-k_1k_2k_3-2=2$.
\end{lemma}

\begin{proof}
This follows by combining the observation that two matrices $A,B\in\SL_2(\C)$ share an eigenvector if and only if $\tr(ABA^{-1}B^{-1})=2$ with the expression for this trace in terms of $\tr(A)$, $\tr(B)$, and $\tr(AB)$, derived for instance in Section \ref{sect:2.3.1}.
\end{proof}

\subsubsection{Surfaces of type $(1,1)$} \label{sect:3.1.1}
Suppose $\Sigma$ is of type $(1,1)$, and let $(\alpha,\beta,\gamma)$ be optimal generators of $\pi_1\Sigma$. The moduli space $X_k$ can be presented as an affine cubic algebraic surface, with equation
$$x^2+y^2+z^2-xyz-2=k$$
where the variables $(x,y,z)$ correspond to the monodromy traces $(\tr_{\alpha},\tr_{\beta},\tr_{\alpha\beta})$ and $\tr_{\gamma}\equiv k$. Now, let $a\subset\Sigma$ be the essential curve underlying $\alpha$. Since any essential curve of $\Sigma$ is, up to isotopy, mapping class group equivalent to $a$, it suffices to prove Theorem \ref{linethm}.A for $(g,n)=(1,1)$ where $P=a$. Let $t\in X(a,\C)\simeq\C$. (We are following the notation of Section \ref{sect:2.2.2}; in effect, this means we are fixing the trace $t$ of monodromy along $a$.) Since $X_{k,t}^P=X_{k,t}^a$ is a conic section
$$y^2-tyz+z^2+t^2-2-k=0$$
in the $(y,z)$-plane, elementary geometry shows that $X_{k,t}^a$ is parabolic if and only if
(1) $t=\pm2$, or (2) the conic section is degenerate. Let us assume $t\neq\pm2$; the latter condition states that the equation for $X_{k,t}^a$ factors as
$$y^2-tyz+z^2+t^2-2-k=(y-\lambda z+m_1)(y-\lambda^{-1} z+m_2)=0$$
for some $\lambda\in\C^*$ and $m_i\in\C$. Expanding and comparing coefficients, we must have
\begin{itemize}
	\item $\lambda+\lambda^{-1}=t$,
	\item we have
	$$\begin{bmatrix}1 & 1\\ -\lambda & -\lambda^{-1}\end{bmatrix}\begin{bmatrix}m_1\\ m_2\end{bmatrix}=\begin{bmatrix}0 \\0\end{bmatrix}$$
	and hence $m_1=m_2=0$, and therefore
	\item $m_1m_2=t^2-2-k=0$, and in particular $k\neq2$.
\end{itemize}
Thus, we see that $X_{k,t}^{a}$ is parabolic if and only if
\begin{enumerate}
	\item[\textup{(1)}] $t=\pm2$, or
	\item[\textup{(2)}] $t\neq\pm2$, $(g,n,k)\neq(1,1,2)$, and $(X_{k,t}^a)|_{\Sigma|a}=(t,t,t^2-2)$.
\end{enumerate}
Here, we made an identification $X(\Sigma|a)\simeq\A^3$. By Lemma \ref{redlem}, the last condition in (2) is equivalent to saying that $(X_{k,t}^a)|_{\Sigma|a}$ is reducible. This concludes the proof of Theorem \ref{linethm}.A for $(g,n)=(1,1)$. Combining this with Lemma \ref{linelem}, we immediately obtain Theorem \ref{linethm}.B in this case.

\subsubsection{Surfaces of type $(0,4)$} \label{sect:3.1.2}
Suppose $\Sigma$ is of type $(0,4)$, and let $(\gamma_1,\gamma_2,\gamma_3,\gamma_4)$ be optimal generators of $\pi_1\Sigma$. The moduli space $X_k$ is an affine cubic algebraic surface with equation
$$x^2+y^2+z^2+xyz-Ax-By-Cz-D=0$$
where the variables $(x,y,z)$ correspond to the traces $(\tr_{\gamma_1\gamma_2},\tr_{\gamma_2\gamma_3},\tr_{\gamma_1\gamma_3})$ and
$$\left\{\begin{array}{l}A=k_1k_2+k_3k_4\\ B=k_2k_3+k_1k_4\\ C=k_1k_3+k_2k_4\end{array}\right.\quad\text{and}\quad D=4-k_1^2-k_2^2-k_3^2-k_4^2-k_1k_2k_3k_4$$
with $k_i\equiv\tr\gamma_i$. Now, let $a,b,c\subset\Sigma$ be essential curves lying in the free homotopy classes of $\gamma_1\gamma_2$, $\gamma_2\gamma_3$, and $\gamma_1\gamma_3$, respectively. Since any essential curve of $\Sigma$ is, up to isotopy, mapping class group equivalent to one of the curves $a$, $b$, and $c$, it suffices to prove Theorem \ref{linethm}.A for $(g,n)=(0,4)$ where $P$ consists of one of the curves $a$, $b$, and $c$. We treat the case $P=a$ in what follows; the remaining cases will proceed similarly. Let $t\in X(a,\C)\simeq\C$. Again by elementary geometry, we see that $X_{k,t}^a$ is parabolic if and only if (1) $t=\pm2$, or (2) $X_{k,t}^a$ is a degenerate conic in the $(y,z)$-plane. Let us assume $t\neq\pm2$; the latter condition states that the equation for $X_{k,t}^a$ factors as
$$t^2+y^2+z^2+tyz-At-By-Cz-D=(y+\lambda z+m_1)(y+\lambda^{-1}z+m_2)=0$$
for some $\lambda\in\C^*$ and $m_i\in\C$. Expanding and comparing coefficients, we see that
\begin{itemize}
	\item $\lambda+\lambda^{-1}=t$,
	\item we have 
	$$\begin{bmatrix}1 & 1\\\lambda & \lambda^{-1}\end{bmatrix}\begin{bmatrix}m_1\\ m_2\end{bmatrix}=\begin{bmatrix}-B\\ -C\end{bmatrix},$$
	and
	\item $m_1m_2=-At-D$.
\end{itemize}
This is equivalent to
$$-\frac{B^2-tBC+C^2}{t^2-4}=\frac{1}{(\lambda^{-1}-\lambda)^2}(\lambda^{-1}B-C)(-\lambda B+C)=-At-D$$
or in other words $(t^2-4)(-Ax-D)+(B^2-tBC+C^2)=0$. Upon rearranging, this is seen to be equivalent to
$$(k_1^2+k_2^2+x^2-tk_1k_2-4)(k_3^2+k_4^2+x^2-tk_3k_4-4)=0,$$
which in turn is equivalent to saying that at least one factor of $(X_{k,t}^a)|_{\Sigma|a}$ is reducible, by Lemma \ref{redlem}. This proves Theorem \ref{linethm}.A for $(g,n)=(0,4)$, and Theorem \ref{linethm}.B follows by Lemma \ref{linelem}.

\subsection{General case} \label{sect:3.2}
Let $\Sigma$ be a surface of type $(g,n)$ with $3g+n-3>0$, and let $X_k=X_k(\Sigma)$ be a relative character variety of $\Sigma$. We shall first prove the following claim, by induction on $(g,n)$.	Claim \ref{firstclaim} proves Theorem \ref{linethm}.B, conditional on Theorem \ref{linethm}.A.

\begin{claim}
\label{firstclaim}
For any nonconstant morphism $f:\A^1\to X_k$, there exists an imperfect fiber $X_{k,t}^P$ for some $(P,t)$ containing the image of $f$.
\end{claim}

\begin{proof}
We have already proved Theorem \ref{linethm} for $(g,n)=(1,1)$ and $(0,4)$, and these will provide the base cases for our induction. Let $f:\A^1\to X_k$ be a nonconstant morphism. By Lemma \ref{linelem}, there is an essential curve $a\subset\Sigma$ such that $\tr_{a}\circ f\equiv t_0$ is constant. Suppose first that the restriction
$$f|_{\Sigma'}:\A^1\to X_k(\Sigma)\xrightarrow{(-)|_{\Sigma'}}X(\Sigma')$$
of $f$ to some connected component $\Sigma'$ of $\Sigma|a$ is nonconstant. We observe that the image of $f|_{\Sigma'}$ lies in $X_{k'}(\Sigma')$ for some $k'\in X(\del\Sigma',\C)$ determined by $k$ and $t_0$. Hence, by inductive hypothesis, there is a pants decomposition $P'$ of $\Sigma'$ and an element $t'\in X(P',\C)$ such that $X_{k',t'}^{P'}(\Sigma')$ is an imperfect fiber containing the image of $f|_{\Sigma'}$. If $a$ is nonseparating, then taking
$$(P,t)=(P'\sqcup a, (t', t_0))$$
we see that $X_{k,t}^P$ is an imperfect fiber containing the image of $f$, as desired. If $f$ is separating and $\Sigma|a=\Sigma'\sqcup\Sigma''$, then again $f|_{\Sigma''}$ has image lying in some $X_{k''}(\Sigma'')$ for some $k''\in X(\del\Sigma'',\C)$ determined by $k$ and $t_0$, and by a repeated application of Lemma \ref{linelem} we see that there is a pants decomposition $P''$ of $\Sigma''$ and $t''\in X(P'',\C)$ such that $X_{k'',t''}^{P''}$ contains the image of $f|_{\Sigma''}$. Taking
$$(P,t)=(P'\sqcup a\sqcup P'',(t',t_0,t'')),$$
we again see that $X_{k,t}^P$ is an imperfect fiber containing the image of $f$.

To complete our proof of the claim, it remains only to consider the case where $f|_{\Sigma|a}$ is constant.

We first consider the case where $a$ is nonseparating. Let $\alpha_1,\cdots,\alpha_{2g+n}$ be a sequence of optimal generators for $\pi_1\Sigma$. Up to mapping class group action, we may assume that $a$ is the essential curve underlying $\alpha_1$. By Fact \ref{fact}, the coordinate ring of the fiber $X_k\to X(\Sigma|a)$ above $f|_{\Sigma|a}$ is generated by functions of the form
$$\tr_{\alpha_2},\tr_{\alpha_2\alpha_{i}},\tr_{\alpha_2\alpha_i\alpha_j},\tr_{\alpha_1\alpha_2},\tr_{\alpha_1\alpha_2\alpha_{i}}$$
for $3\leq i<j\leq 2g+n$. Thus, since $f$ is nonconstant, the composition of $f$ with at least one of the above coordinate functions must be nonconstant. Let us consider the case where $\tr_{\alpha_2}\circ f$ is nonconstant; the other cases will follow similarly. There is a surface $\Sigma'\subset\Sigma$ of type $(1,1)$ containing the loops $\alpha_1$ and $\alpha_2$ in its interior. By our hypothesis, we see that $f|_{\Sigma'}$ is nonconstant, and the image of $f|_{\Sigma'}$ lies in $X_{k'}(\Sigma)$ for some $k'\in X(\del\Sigma',\C)$ determined by the (constant) value of $f|_{\Sigma|a}$; indeed, the boundary of $\Sigma'$ lies in $\Sigma|a$. Thus, by the case of Theorem \ref{linethm} for $(g,n)=(1,1)$ proved above, there is an essential curve $a'\subset\Sigma'$ and $t'\in X(a',\C)$ such that $X_{k',t'}^{a'}(\Sigma')$ is an imperfect fiber containing the image of $f|_{\Sigma'}$. Completing $a'\sqcup\del\Sigma'\subset\Sigma$ to a pants decomposition $P$ of $\Sigma$, we see that there is some $t\in X(P,\C)$ such that $X_{k,t}^P$ is an imperfect fiber containing the image of $f$, as desired.

The case where $a$ is separating is very similar, by appropriately invoking the case of Theorem \ref{linethm} for $(g,n)=(0,4)$.
\end{proof}

Suppose that $P=a_1\sqcup\cdots\sqcup a_{3g+n-3}$ is a pants decomposition of $\Sigma$, and suppose that $t=(t_1,\cdots,t_{3g+n-3})\in X(P,\C)\simeq\C^{3g+n-3}$ is chosen so that $X_{k,t}^P$ is a perfect fiber. Let us consider the morphism
$$X_{k,t}^P(\Sigma)\to\prod_{i=1}^{3g+n-3}X_{k_i,t_i}^{a_i}(\Sigma_i)$$
where each $\Sigma_i$ is the surface of type $(0,4)$ or $(1,1)$ obtained by gluing together the two boundary curves on $\Sigma|P$ corresponding to $a_i$, and the boundary traces $k_i$ are appropriately determined from $k$, $P$, $t$, and $\Sigma_i$. As a consequence of Proposition \ref{perfprop} proved in Section \ref{sect:4.3}, the above morphism is finite at the level of complex points (cf.~proof of Corollary \ref{perfectcor}). Combining this with the results from Section \ref{sect:3.1}, we deduce that there cannot be a nonconstant morphism from $\A^1$ to such a perfect fiber, proving one half of Theorem \ref{linethm}.A.

We shall henceforth assume that $(g,n)\neq(1,1),(0,4)$, to simplify our remaining argument. To complete the proof of Theorem \ref{linethm}, it suffices to prove the following.

\begin{claim}
\label{secondclaim}
Assume that $(g,n)\neq(1,1),(0,4)$. Given a semisimple representation $\rho:\pi_1\Sigma\to\SL_2(\C)$ whose image in the character variety $X(\Sigma)$ lies in an imperfect fiber $X_{k,t}^P$, there is a one-parameter polynomial family
$$\rho_T:\pi_1\Sigma\to\SL_2(\C)$$
of representations with nonconstant images all lying in $X_{k,t}^P$, so that we have $\rho=\rho_{T_0}$ for some $T_0\in\C$.
\end{claim}

\begin{proof}
Let $\rho$ and $X_{k,t}^P$ be as above. We shall argue by division into several cases. For the benefit of the reader, we list the cases broadly considered and their hypotheses.
\begin{enumerate}
	\item \textbf{Parabolic curve.} There is a curve $a\subset P$ with $\tr_{a}(X_{k,t}^P)=\pm2$.
	\begin{itemize}
		\item \textsc{Case 1:} the curve $a$ is separating.
		\item \textsc{Case 2:} the curve $a$ is nonseparating.
	\end{itemize}
	\item \textbf{Parabolic pants.} There is no curve $a\subset P$ with trace $\pm2$, but there is a component $\Sigma'$ of $\Sigma|P$ such that $X_{k,t}^P|\Sigma'$ is reducible.
	\begin{itemize}
		\item \textsc{Case 1:} the image of $\Sigma'$ in $\Sigma$ is a surface of type $(1,1)$.
		\item \textsc{Case 2:} the image of $\Sigma'$ in $\Sigma$ is a surface of type $(0,3)$.
	\end{itemize}
\end{enumerate}
We now begin our proof.
\\

\noindent\textbf{Parabolic curve.} Let us first consider the case where there is a curve $a\subseteq P$ with $\tr_a(X_{k,t}^P)=\pm2$. We may assume to have fixed the base point of $\Sigma$ to lie on $a$. Let $\alpha$ be a smooth simple loop parametrizing $a$. Up to global conjugation, we may assume that
\begin{align}
\tag{$*$}\rho(\alpha)=s\begin{bmatrix} 1 & u\\0 & 1\end{bmatrix}
\end{align}
for $s\in\{\pm1\}$ and $u\in\{0,1\}$. There are several elementary cases to consider.

\textsc{Case 1:} the curve $a$ is separating. Let us write $\Sigma|a=\Sigma_1\sqcup\Sigma_2$. Up to conjugation, we may assume that on top of $(*)$ the following conditions hold:
\begin{enumerate}
	\item[\textup{(1)}] $\rho|\Sigma_1$ is irreducible or upper triangular,
	\item[\textup{(2)}] $\rho|\Sigma_2$ is irreducible or lower triangular, and
	\item[\textup{(3)}] if $\rho|\Sigma_1$ and $\rho|\Sigma_2$ are both reducible, then they are either both non-diagonal or both diagonal.
\end{enumerate}
Indeed, if one of $\rho|\Sigma_1$ and $\rho|\Sigma_2$ is irreducible, then up to relabeling $\Sigma_1$ and $\Sigma_2$ we may assume $\rho|\Sigma_2$ is irreducible, and $\rho|\Sigma_1$ must be irreducible or upper triangular up to conjugation. So suppose both $\rho|\Sigma_1$ and $\rho|\Sigma_2$ are reducible. This implies that $u=0$ in ($*$) above, since otherwise $\rho$ must be upper triangular and non-diagonal, contradicting the hypothesis that $\rho$ is semisimple. Unless $\rho$ is reducible (whence diagonal), there is a basis $v_1,v_2$ of $\C^2$ such that each $v_i$ is a common eigenvector for $\rho|\Sigma_i$. Up to conjugation $M^{-1}\rho M$ of $\rho$ by the invertible matrix $M=[v_1,v_2]$, we may thus assume $\rho|\Sigma_1$ is upper triangular and $\rho|\Sigma_2$ is lower triangular. For convenience, we shall denote $\rho_i=\rho|\Sigma_i$ so that we may write $\rho=(\rho_1,\rho_2)$ using the notation of the second part of Section \ref{sect:2.2.3}.

\textsc{Subcase 1a:} $\rho_2$ is non-diagonal. Let us consider the family of representations $\rho^T=(\rho_1,u_T\rho_2 u_T^{-1})$ for $T\in\C$ where
	$$u_T=\begin{bmatrix}1 & T\\ 0 & 1\end{bmatrix}.$$
Note that $\rho_1(\alpha)=u_T\rho_2(\alpha)u_T^{-1}$ so the representation $\rho^T$ is well-defined.
For any $\beta\in\pi_1\Sigma_1$ and $\gamma\in\pi_1\Sigma_2$ with
$$\rho(\beta)=\begin{bmatrix} b_1 & b_2 \\ b_3 & b_4\end{bmatrix}\quad\text{and}\quad \rho(\gamma)=\begin{bmatrix} c_1 & c_2 \\ c_3 & c_4\end{bmatrix}
$$
we have
\begin{align*}
\tr\rho^T(\beta\gamma)&=\tr\left(\begin{bmatrix} b_1 & b_2 \\ b_3 & b_4\end{bmatrix}\begin{bmatrix}1 & T\\ 0 & 1\end{bmatrix}\begin{bmatrix} c_1 & c_2 \\ c_3 & c_4\end{bmatrix}\begin{bmatrix}1 & -T\\ 0 & 1\end{bmatrix}\right)\\
&=b_1c_2+b_2c_3+b_3c_2+b_4c_4+(b_1c_3-b_3c_1-b_4c_3+b_3c_4)T-b_3c_3T^2.
\end{align*}
Since $\rho_2$ is non-diagonal while irreducible or lower triangular, we may choose $\gamma$ as above so that $c_3\neq0$. We have the following possibilities.
\begin{itemize}
	\item (a) Suppose $\rho_1$ is irreducible. We can choose $\beta$ as above with $b_3\neq0$, so that $\tr_{\beta\gamma}(\rho^T)$ is a nonconstant function of $T$.
	\item (b) Suppose $\rho_1$ is upper triangular (so $b_3=0$ for any choice of $\beta$), and there exists $\beta\in\pi_1\Sigma_1$ such that $\tr\rho(\beta)\neq\pm2$. Choosing such $\beta$ we find that $\tr_{\beta\gamma}(\rho^T)$ is a nonconstant function of $T$ since $(b_1-b_4)c_3\neq0$.
	\item (c) Finally, consider the case where $\rho_1$ is upper triangular and $\tr\rho(\beta)=\pm2$ for any $\beta\in\pi_1\Sigma_1$. This implies that the image of $\rho_1$ is abelian. Suppose $\Sigma_1$ is of type $(h,m)$, and let $S=(\alpha_1,\beta_1,\dots,\alpha_h,\beta_h,\gamma_1,\dots,\gamma_m)$ be an optimal sequence of generators for $\pi_1\Sigma_1$ such that $\gamma_m=\alpha$ (up to homotopy). Let us define a one-parameter family $\rho_1^T$ of upper triangular deformations of $\rho_1=\rho_1^0$ given by setting
	$$\left\{\begin{array}{ll}\rho_1^T(\alpha_1)=\rho_1(\alpha_1)u_T,&\text{and}\\\rho_1^T(\ell)=\rho_1(\ell) &\text{for any other $\ell\in S$}\end{array}\right.$$
	if $h\geq1$, and setting
	$$\left\{\begin{array}{ll}\rho_1^T(\gamma_1)=\rho_1(\gamma_1)u_T,&\\\rho_1^T(\gamma_2)=u_T^{-1}\rho_1(\gamma_2),&\text{and}\\\rho_1^T(\ell)=\rho_1(\ell) &\text{for any other $\ell\in S$}\end{array}\right.$$
	if $h=0$ so that $m\geq3$. Then choosing $\beta=\alpha_1$ (resp.~$\beta=\gamma_1$) if $h\geq1$ (resp.~$h=0$) we find that
	$$\tr_{\beta\gamma}(\rho^T)=\tr\rho(\beta\gamma)\pm c_3T$$
	which is a nonconstant function of $T$.
\end{itemize}
Thus, in each of the cases the morphism $\A^1\to X_{k,t}^P$ given by $T\mapsto\rho^T$ is nonconstant.

\textsc{Subcase 1b:} $\rho_1$ is non-diagonal. This case is established by the same argument as in \textsc{Subcase 1a}. The only difference is that, instead of the matrices $u_T$, we consder in appropriate places of our argument the matrices
$$l_T=\begin{bmatrix}1 & 0 \\ T & 1\end{bmatrix}.$$

\textsc{Subcase 1c:} Both $\rho_1$ and $\rho_2$ are diagonal. We shall first construct a nontrivial family of upper triangular representations $\rho_1^T$ of $\pi_1\Sigma_1$ with $\rho_1^0=\rho_1$ and $\rho_1^T(\alpha)=\rho_1(\alpha)$ for all $T$ as follows.
\begin{itemize}
	\item If there exists $\beta\in\pi_1\Sigma_1$ such that $\tr\rho(\beta)\neq\pm2$, then let $\rho_1^T=u_T\rho_1 u_T^{-1}$.
	\item If $\tr\rho_1(\beta)=\pm2$ for all $\beta\in\pi_1\Sigma$, then define $\rho_1^T$ as in the treatment of possibility (c) in \textsc{Subcase 1a}.
\end{itemize}
Note that, in both cases, there exists $\beta\in\pi_1\Sigma_1$ such that the upper right corner entry of $\rho_1^T(\beta)$ is a nonconstant polynomial function of $T$. Similarly, let us construct a nontrivial family of lower triangular representations $\rho_2^T$ of $\pi_1\Sigma_2$ with $\rho_2^0=\rho_2$ and $\rho_2^T(\alpha)=\rho_2(\alpha)$ for all $T$, in such a way that there exists $\gamma\in\pi_1\Sigma_2$ such that the lower left corder entry of $\rho_2^T(\gamma)$ is a nonconstant polynomial of $T$.

Finally, let us define the representation $\rho^T=(\rho_1^T,\rho_2^T)$, which makes sense since we have $\rho_1^T(\alpha)=\rho(\alpha)=\rho_2^T(\alpha)$ for all $T$ by construction. For $\beta$ and $\gamma$ chosen as above, we see that $\tr_{\beta\gamma}(\rho^T)$ is a nonconstant polynomial in $T$. Thus the morphism $\A^1\to X_{k,t}^P$ defined by $T\mapsto\rho^T$ is nonconstant, passes through $\rho$.

\textsc{Case 2:} the curve $a$ is nonseparating. We shall write $\rho=(\rho|(\Sigma|a),\rho(\beta))=(\rho',B)$ using the notation of first part of Section \ref{sect:2.2.3}, with a choice of simple loop $\beta$ intersecting $\alpha$ exactly once. Up to conjugation, we may assume that the representation $\rho'$ is irreducible or upper triangular.

\textsc{Subcase 2a:} $\rho'$ is irreducible or $B$ is not upper triangular. Let us consider the family of representations
$$\rho^T=(\rho',B_T)$$
where
$$B_T=u_TB=\begin{bmatrix}1 & T\\ 0 & 1\end{bmatrix}B.$$
Note that $\rho^T$ is well-defined since $B^{-1}u_T^{-1}\rho(\alpha_1)u_TB=B^{-1}\rho(\alpha_1)B=\rho(\alpha_2)$ in the notation of \ref{sect:2.2.3}. Now, consider the morphism $f:\A^1\to X_{k,t}^P$ given by $T\mapsto \rho^T$. Note that we have $\rho^0=\rho$. We claim that this morphism is nonconstant. To see this, it suffices to show that there is some element $\gamma\in\pi_1(\Sigma)$ where $\tr_\gamma\circ f:\A^1\to\A^1$ is nonconstant. We have two possibilities.
\begin{itemize}
	\item If $B$ is not upper triangular, then $\gamma=\beta$ suffices.
	\item If $B$ is upper triangular but $\rho'$ is irreducible, then there exists $\delta\in\pi_1(\Sigma|a)$ which is not upper triangular. It suffices to choose $\gamma=\beta\delta$.
\end{itemize}

\textsc{Subcase 2b:} $\rho'$ and $B$ are both upper triangular, so that $\rho$ is reducible. Since $\rho$ is semisimple, $\rho$ must be diagonal and in particular $\rho(\alpha)=\pm\mathbf1$. Let $\Sigma_1\subset\Sigma$ be the subsurface of type $(1,1)$ obtained by taking a closed tubular neighborhood of $a\cup b$ where $b$ is the curve underlying $\beta$. Let $c$ be the boundary curve of $\Sigma_1$, and write $\Sigma|c=\Sigma_1\sqcup\Sigma_2$ where $\Sigma_2$ is a surface of type $(g-1,n+1)$. For convenience, we shall denote $\rho_i=\rho|\Sigma_i$.

Let us write $\rho_1=(\rho_1',B)$ in the notation of Section \ref{sect:2.2.3} (with the same choice of $\alpha$ and $\beta$ as before), and consider the family of lower triangular representations $\rho_1^T=(\rho_1',B_T)$ where
$$B_T=l_TB=\begin{bmatrix}1 & 0\\ T & 1\end{bmatrix}B.$$
Note that the lower left entry of $B_T$ is a nonconstant function of $T$. Now, without loss of generality, we may assume that our new basepoint lies on $c$. Let $\gamma$ be a simple loop parametrizing $c$, so that $\rho(\gamma)=\mathbf 1$ and $\rho_1^T(\gamma)$ is constant for all $T$. Proceeding as in \textsc{Case 1}, we can construct a nonconstant upper triangular deformation $\rho_2^T$ of $\rho_2$ with $\rho_2^0=\rho_2$ such that $\rho_2^T(\gamma)=\mathbf1$ for all $T$:
\begin{itemize}
	\item If there exists $\gamma\in\pi_1\Sigma_2$ such that $\tr\rho(\gamma)\neq\pm2$, then $\rho_2^T=u_T\rho u_T^{-1}$.
	\item If $\tr\rho_2(\beta)=\pm2$ for all $\beta\in\pi_1\Sigma$, then define $\rho_2^T$ as in the treatment of possibility (c) in \textsc{Subcase 1a} of \textsc{Case 1}.
\end{itemize}
Then the representation $\rho^T=(\rho_1^T,\rho_2^T)$ (in the notation of second part of Section \ref{sect:2.2.3}) is well-defined and has the property that $\rho^0=\rho$. Since there exists $\gamma\in\pi_1\Sigma_2$ such that the upper right entry of $\rho^T(\gamma)$ is nonconstant, the nonconstancy of consideration of $\tr_{\beta\gamma}(\rho^T)$ shows that the morphism $f:\A^1\to X_k$ given by $T\mapsto \rho^T$ is nonconstant. Moreover, since $\rho^T|(\Sigma|a)$ remains upper triangular, we see that the image of $f$ lies in $X_{k,t}^P$, as desired.
\\

\noindent \textbf{Parabolic pants.} We now consider the case where the following conditions hold:
\begin{enumerate}
	\item[\textup{(1)}] $\tr_a(X_{k,t}^P)\neq\pm2$ for every curve $a\subseteq P$,
	\item[\textup{(2)}] $(g,n,k)\neq(1,1,2)$, and
	\item[\textup{(3)}] $(X_{k,t}^P)|_{\Sigma'}$ is reducible for some connected component $\Sigma'$ of $\Sigma|P$.
\end{enumerate}
We may assume for convenience that the base point $x\in\Sigma$ lies on $\Sigma'$. Let $\gamma_1,\gamma_2,\gamma_3$ be optimal generators for $\pi_1\Sigma'$ corresponding to the boundary curves $c_1,c_2,c_3$ of $\Sigma'$. We shall write $\rho|_{\Sigma'}=(t_1,t_2,t_3)\in X(\Sigma')\simeq\A^3$ (see Example \ref{exfree}.(2)). By relabeling the boundary curves of $\Sigma'$ if necessary, we may assume that $c_1$ corresponds to a curve of $P$. In particular, $t_1\neq\pm2$ by our hypothesis above. We further assume that, if the image of $\Sigma'$ in $\Sigma$ is a surface of type $(1,1)$, then $c_1$ and $c_2$ map to the same curve in $P$.

It will be convenient for us to introduce distinguished families of representations $\pi_1\Sigma'\to\SL_2(\C)$ which are reducible. For each $T\in\C$ and $s\in\{\pm1\}$, let $\rho_T^s:\pi_1\Sigma'\to\SL_2(\C)$ be the representation determined by
$$\rho_T^s(\gamma_1)=\begin{bmatrix}\lambda & 0\\ 0 & \lambda^{-1}\end{bmatrix},\quad\rho_T^s(\gamma_2)=\begin{bmatrix}\mu & T\\ 0 & \mu^{-1}\end{bmatrix},$$
where $\lambda\in\C^\times\setminus\{\pm1\}$ and $\mu\in\C^\times$ are such that
$$\lambda^{s}\in\{z\in\C^\times:\Im(z)\geq0,z\notin[-1,1]\}$$
and
$$\left\{\begin{array}{l}t_1=\lambda+\lambda^{-1},\\t_2=\mu+\mu^{-1},\\t_3=\lambda\mu+\lambda^{-1}\mu^{-1}.\\\end{array}\right.$$
(The sign $s$ is put there simply to remove amibiguities; they do not play a significant role in the proof.) Note that the Jordan equivalence class of each $\rho_T^s$ in $X(\Sigma')$ is equal to that of $\rho|_{\Sigma'}$ since a point of $X(\Sigma')\simeq\A^3$ is determined by its traces along the boundary curves of $\Sigma'$. Up to global conjugation, we may assume that $\rho|_{\Sigma'}=\rho_{T_0}^s$ for some $T_0\in\C$ and $s\in\{\pm1\}$. Below, we proceed with a fixed choice of $s\in\{\pm1\}$, as it will not make a difference to the argument. We shall write also $\rho_{T}^{s}=\rho_{T}'$ for easier notation.

We must consider different cases, according to the relative position of $\Sigma'$ in $\Sigma$.

\textsc{Case 1:} the image of $\Sigma'$ in $\Sigma$ is a surface of type $(1,1)$. By our hypothesis, the boundary curves $c_1$ and $c_2$ map to the same curve $a\subset P$, while $c_3$ maps to a separating curve $c\subset\Sigma$. Let us write $\Sigma|c=\Sigma''\sqcup\Sigma'''$ with $\Sigma''$ being the image of $\Sigma'$. Without loss of generality, we may assume that the implicit base point $x\in\Sigma$ lies on $c$. Let $(\alpha,\beta,\gamma)$ be a sequence of optimal generators for $\pi_1\Sigma''$, such that under the immersion $\Sigma'\to\Sigma''$ we have
$$\gamma_1\mapsto\alpha,\quad\gamma_2\mapsto\beta^{-1}\alpha^{-1}\beta,\quad\gamma_3\mapsto\gamma.$$
Let us write
$$\rho(\beta)=\begin{bmatrix}B_1 & B_2\\ B_3 & B_4\end{bmatrix}.$$
The condition $\rho(\gamma_2)=\rho(\beta^{-1}\alpha^{-1}\beta)$ is then
\begin{align*}
\begin{bmatrix}\lambda & T_0\\ 0 & \lambda^{-1}\end{bmatrix}&=\begin{bmatrix} B_4 & -B_2\\ -B_3 & B_1\end{bmatrix}\begin{bmatrix}\lambda^{-1} & 0\\ 0 & \lambda\end{bmatrix}\begin{bmatrix} B_1 & B_2\\ B_3 & B_4\end{bmatrix}\\
&=\begin{bmatrix}\lambda+(\lambda^{-1}-\lambda)B_1B_4 &(\lambda^{-1}-\lambda)B_2B_4\\(\lambda-\lambda^{-1})B_1B_3 & \lambda^{-1}+(\lambda-\lambda^{-1})B_1B_4\end{bmatrix}
\end{align*}
where we have $\lambda-\lambda^{-1}\neq0$ by our hypothesis on $\rho$ that $\tr\rho(\alpha)\neq\pm2$. This shows that we must have
$$\rho(\beta)=\begin{bmatrix}0 & B_2\\  -B_2^{-1}&\tr\rho(\beta)\end{bmatrix}$$
where furthermore $(\lambda^{-1}-\lambda)B_2\tr\rho(\beta)=T_0$. Up to global conjugation of $\rho$ by a diagonal matrix (which also results in a suitable adjustment of $T_0$), we may further assume that $B_2=1$, and hence $\tr\rho(\beta)=T_0/(\lambda^{-1}-\lambda)$. Let $\rho_T'':\pi_1\Sigma''\to\SL_2(\C)$ be the representation determined by
$$\rho_T''(\alpha)=\begin{bmatrix}\lambda & 0\\ 0 &\lambda^{-1}\end{bmatrix},\quad\rho_T''(\beta)=\begin{bmatrix}0 & 1\\ -1 & \frac{T}{\lambda^{-1}-\lambda}\end{bmatrix}.$$
The preceding observations show that $\rho_{T_0}''=\rho|_{\Sigma''}$ and $\rho_T''|_{\Sigma'}=\rho_T'$ for every $T\in\C$. Note that we have
$$\rho_T''(\gamma)=\rho_T'(\gamma_3)=\begin{bmatrix}\lambda^{-2} & -\lambda T\\ 0 & \lambda^2\end{bmatrix}.$$
For $T\in\C$, let $\rho_T:\pi_1\Sigma\to\SL_2(\C)$ be the representation such that $\rho_T|\Sigma'''=\rho$ and
$$\rho_T|\Sigma''=\begin{bmatrix}1 & \lambda\frac{T-T_0}{\lambda^2-\lambda^{-2}} \\0 & 1\end{bmatrix}\rho_T''\begin{bmatrix}1 & \lambda\frac{T-T_0}{\lambda^2-\lambda^{-2}} \\0 & 1\end{bmatrix}^{-1}.$$
Here, $\lambda^2-\lambda^{-2}\neq0$ since otherwise $\tr\rho(\gamma)=\lambda^{-2}+\lambda^2=\pm2$, which was precluded. It can be directly verified that
$$\rho(\gamma)=\begin{bmatrix}1 & \lambda\frac{T-T_0}{\lambda^2-\lambda^{-2}} \\0 & 1\end{bmatrix}\rho_T''(\gamma)\begin{bmatrix}1 & \lambda\frac{T-T_0}{\lambda^2-\lambda^{-2}} \\0 & 1\end{bmatrix}^{-1}$$
so $\rho_T$ is well-defined by our discussion in Section \ref{sect:2.2.3}. We have $\rho=\rho_{T_0}$ by construction, and we see that the morphism $\A^1\to X_k$ given by $T\mapsto\rho_T$ lies in the fiber $X_{k,t}^P$. The fact that this morphism is nonconstant can be deduced from the observation that
$$\tr\rho_T(\beta)=\frac{T}{\lambda^{-1}-\lambda}$$
is a nonconstant function of $T$. Therefore, this proves the claim when the image of $\Sigma'$ in $\Sigma$ is a surface of type $(1,1)$.

\textsc{Case 2:} the boundary curves of $\Sigma'$ map to three distinct curves in $\Sigma$ under the immersion $\Sigma'\to\Sigma$ (which we shall also denote $c_1,c_2,c_3$ for simplicity). Now, let us write
\begin{align*}
\left\{\begin{array}{rl}\Sigma|c_3&=\Sigma_3\sqcup\Sigma_3^{\circ}\\\Sigma_3|c_2&=\Sigma_2\sqcup\Sigma_2^{\circ}\\
\Sigma_2|c_1&=\Sigma_1\sqcup\Sigma_1^\circ\end{array}\right.
\end{align*}
where $\Sigma_3$ is the connected component of $\Sigma|c_3$ containing $\Sigma'$ (here, $\Sigma_3^\circ$ is empty if $c_3$ is nonseparating or is a boundary curve in $\Sigma$), $\Sigma_2$ is the connected component of $\Sigma_3|c_2$ containing $\Sigma'$, and finally $\Sigma_1=\Sigma'$.

Let us extend the polynomial family of representations $\rho_T':\pi_1\Sigma'\to\SL_2(\C)$ to the polynomial family $\rho_T'':\pi_1\Sigma_2\to\SL_2(\C)$ given by
$$\rho_T''|_{\Sigma'}=\rho_{T}',\quad \rho_T''|_{\Sigma_1^\circ}=\rho|_{\Sigma_1^\circ}.$$
This is well-defined by our discussion in Section \ref{sect:2.2.3}. Note that $\rho_{T_0}''=\rho|_{\Sigma_2}$. Our next step is to extend $\rho_T''$ to a family $\rho_T''':\pi_1\Sigma_3\to\SL_2(\C)$ such that $\rho_{T_0}'''=\rho|\Sigma_3$. We need to consider three cases.
\begin{enumerate}
	\item[\textup{(1)}] Suppose $c_2$ is a boundary curve on $\Sigma_3$, so that $\Sigma_2=\Sigma_3$. Let $\rho_T'''=\rho_T''$.
	\item[\textup{(2)}] Suppose $c_2$ is a separating curve on $\Sigma_3$. We define $\rho_T'''$ by requiring
	$$\rho_T'''|_{\Sigma_2}=\rho_T'',\quad \rho_T'''|_{\Sigma_2^\circ}=\begin{bmatrix}1 & \frac{T-T_0}{\mu^{-1}-\mu}\\0 & 1\end{bmatrix}\rho|_{\Sigma_2^\circ}\begin{bmatrix}1 & \frac{T-T_0}{\mu^{-1}-\mu}\\0 & 1\end{bmatrix}^{-1}.$$
	Note that we must have $\mu^{-1}-\mu\neq0$ since otherwise $\tr\rho(c_2)=\pm2$, which was precluded. By construction, we have
	$$\rho_T''(\gamma_2)=\begin{bmatrix}1 & \frac{T-T_0}{\mu^{-1}-\mu}\\0 & 1\end{bmatrix}\rho(\gamma_2)\begin{bmatrix}1 & \frac{T-T_0}{\mu^{-1}-\mu}\\0 & 1\end{bmatrix}^{-1}=\begin{bmatrix}\mu & T\\ 0 &\mu^{-1}\end{bmatrix},$$
	so that $\rho_T'''$ is well defined by our discussion in Section \ref{sect:2.2.3}.
	\item[\textup{(3)}] Suppose $c_2$ is a nonseparating curve on $\Sigma_3$. Let $\delta$ be a simple based loop on $\Sigma_3$ intersecting $c_2$ exactly once transversally, oriented as in the first part of Section \ref{sect:2.2.3} (with the pair $(\gamma_2,\delta)$ playing the role of $(\alpha,\beta)$ there). Let $\gamma_2'$ be the based loop on $\Sigma_2$ (like $\alpha_2'$ in Section \ref{sect:2.2.3}) whose image in $\Sigma_3$ lies in the homotopy class of $\delta^{-1}\gamma_2\delta$. We define the representation $\rho_T'''$ by specifying the pair $(\rho_T'''|_{\Sigma_2},\rho_T'''(\delta))$ as in the discussion of Section \ref{sect:2.2.3}, where
	$$\rho_T'''|_{\Sigma_2}=\rho_T'',\quad\rho_T'''(\delta)=\begin{bmatrix}1 & \frac{T-T_0}{\mu^{-1}-\mu}\\0 & 1\end{bmatrix}\rho(\delta).$$
	We have $\rho_T'''(\delta^{-1})\rho_T'''(\gamma_2)\rho_T'''(\delta)=\rho(\delta^{-1})\rho(\gamma_2)\rho(\delta)=\rho(\gamma_2')=\rho_T'''(\gamma_2')$ by our construction, so that $\rho_T'''$ is well defined.
\end{enumerate}
Finally, by the same procedure, we define $\rho_T:\pi_1\Sigma\to\SL_2(\C)$ extending $\rho_T'''$ such that $\rho_{T_0}=\rho$. The important point here is that, at each stage of the above ``gluing'' process, the monodromy matrices along the two curves being glued are conjugate by a matrix of the form
$$\begin{bmatrix}1 & P(T) \\0 & 1\end{bmatrix}$$
where $P(T)\in\C[T]$ is a polynomial in $T$. By testing the trace of the representations $\rho^T$ along loops we see that the resulting morphism $\A^1\to X_{k,t}^P$ given by $T\mapsto\rho_T$ must be nonconstant, unless $\rho=\rho_{T_0}$ is reducible.

It therefore remains to treat the case where $\rho$ is reducible. Since $\rho$ is semisimple by hypothesis, it is diagonal. Given $\Sigma'\subset\Sigma|P$ as above, let us choose a different component $\Sigma''\subset\Sigma|P$ whose image in $\Sigma$ has at least one boundary curve (say $c\subset P$) in common with the image of $\Sigma'$ in $\Sigma$. Let $\Sigma_1$ be the surface of type $(0,4)$ obtained by gluing together $\Sigma'$ and $\Sigma''$ along the boundary curves corresponding to $c$. Let us choose our base point of $\Sigma$ to be on $c$, and lift it to $\Sigma_1$. We have a one-parameter family of representations $\rho_T':\pi_1\Sigma_1\to\SL_2(\C)$ whose monodromy along $c$ is constant and is such that $\rho_T'|\Sigma'=\rho_T'$ (with suitable labeling of loops) and $\rho_T'|\Sigma''$ is a similarly constructed polynomial family of lower triangular representations. Note that the morphism $\A^1\to X(\Sigma_1)$ given by $T\mapsto\rho_T'$ is nonconstant since $\tr_{\beta\gamma}(\rho_T)$ is nonconstant for any choice of boundary loops $\beta\in\pi_1\Sigma'$ and $\gamma\in\pi_1\Sigma''$ which remain boundary loops in $\Sigma_1$.

We then proceed as before with the ``gluing'' procedure to produce a family of representations $\rho_T:\pi_1\Sigma\to\SL_2(\C)$ extending $\rho_T'$, the important point being that, at each stage, the monodromy matrices along two curves being glued are conjugate by a matrix which is given by a product of matrices of the form
$$\begin{bmatrix}1 & P(T)\\ 0 & 1\end{bmatrix}\quad\text{or}\quad\begin{bmatrix}1 & 0\\ Q(T) & 1\end{bmatrix}$$
with $P,Q\in\C[T]$. By construction, the morphism $\A^1\to X_{k,t}^P$ given by $T\mapsto\rho_T$ will be nonconstant. This concludes the proof of Claim \ref{secondclaim}
\end{proof}

Claim \ref{secondclaim} implies that, if $X_{k,t}^P$ is imperfect, then it is parabolic. This implies the remaining half of Theorem \ref{linethm}.A, and concludes the proof of Theorem \ref{linethm}.

\subsection{Applications} \label{sect:3.3}
We prove Theorem \ref{theorem1} as a corollary of Theorem \ref{linethm}.

\begin{thm}{1.1}
A point $\rho\in X_k(\C)$ is degenerate if and only if
\begin{enumerate}
	\item \emph{(``parabolic curve'')} there is an essential simple closed curve $a\subset\Sigma$ such that $\tr\rho(a)=\pm2$, or
	\item \emph{(``parabolic pants'')} $(g,n,k)\neq(1,1,2)$ and there is a subsurface $\Sigma'\subset\Sigma$ of genus $0$ with $3$ boundary curves, each of which is an essential curve or a boundary curve of $\Sigma$, such that the restriction $\rho|\Sigma$ is reducible.
\end{enumerate}
In particular, there is a parabolic proper closed subvariety $Z$ of $X_k$ such that every nonconstant morphism $\A^1\to X_k$ over $\C$ is mapping class group equivalent to one with image in $Z$.
\end{thm}

\begin{proof}
The first sentence of the theorem follows directly from Theorem \ref{linethm} and Definition \ref{perfdef}. For each pants decomposition $P$ of $\Sigma$, the condition that $\rho\in X_k(\C)$ lies in some parabolic fiber $X_{k,t}^P$ is evidently nontrivial and algebraic by Theorem \ref{linethm} and the definition of perfect fibers. In particular, the union of all parabolic fibers of the form $X_{k,t}^P$ for fixed $P$ is a proper closed parabolic algebraic subvariety of $X_k$. Since there are at most finitely many isotopy classes of pants decompositions of $\Sigma$ up to mapping class group action, the last statement follows.
\end{proof}

We also record the following consequence of Theorem \ref{linethm}. By the uniformization theorem, given a marked hyperbolic structure $\sigma$ on $\Sigma$ with geodesic boundary curves there is a Fuchsian representation $\rho_\sigma:\pi_1\Sigma\to\SL_2(\R)$ such that the quotient $\H^2/\rho_\sigma(\pi_1\Sigma)$ of the upper half plane contains $(\Sigma,\sigma)$ as a Nielsen core.

\begin{corollary}
\label{polycor}
Every one-parameter polynomial deformation $\rho_t:\pi_1\Sigma\to\SL_2(\R)$ of a Fuchsian representation $\rho_0$ preserving the boundary traces is isotrivial.
\end{corollary}

\begin{proof}
This is immediate from Theorem \ref{linethm} and the observation that a Fuchsian representation $\rho_0$ does not lie in an imperfect fiber $X_{k,t}^P$ for any $(P,t)$.
\end{proof}

\section{Integral points on curves} \label{sect:4}

\subsection{Nonintegrable curves} \label{sect:4.1}
Let $\Sigma$ be a surface of type $(g,n)$ with $3g+n-3>0$, and let $X_k=X_k(\Sigma)$ be a relative character variety of $\Sigma$. We shall prove Theorem \ref{theorem3} using Corollary \ref{corthm1}. We first repeat our definition of integrable curves on $X_k$ given in Section \ref{sect:1.2}.

\begin{definition}
An algebraic curve $C\subseteq X_k$ is \emph{integrable} if there is a pants decomposition $P$ of $\Sigma$ with $\tr_P$ constant along $C$. Otherwise, $C$ is \emph{nonintegrable}.
\end{definition}

As in Section \ref{sect:1.2}, given an algebraic curve $C\subset X_k$ and an arbitrary subset $A\subseteq\C$, let us denote
$$C(A)=\{\rho\in V(\C):\text{$\tr_a(\rho)\in A$ for every essential curve $a\subset\Sigma$}\}.$$

\begin{thm}{1.2}
Let $A\subset\C$ be a closed discrete subset. If $C\subset X_k$ is a nonintegrable geometrically irreducible algebraic curve, then $C(A)$ is finite, with effective bounds on sizes of the coordinates for any given embedding of $C$ into affine space.
\end{thm}

\begin{proof}
Let $F$ be the function field of $C$ over $\C$. Let $\pi:C_0\to C$ be the normalization of $C$, and let $\overline C_0$ be a smooth compactification of $C_0$. Let
$$\{p_1,\cdots,p_m\}=\overline C_0(\C)\setminus C_0(\C)$$
be the points at infinity. For each $p_i$, we have a discrete valuation $v_i$ on $F$ given by order of vanishing at $p_i$. By Corollary \ref{corthm1} and our assumption on $C$, we deduce that there is an essential curve $a_i\subset\Sigma$ such that $v_i(\pi^*({\tr_{a_i}}))\geq0$, meaning in particular that $\pi^*({\tr_{a_i}})$ is bounded on $C_0$ near the point $p_i$. By our hypothesis that $C$ is nonintegrable, we may further assume that each $\tr_{a_i}$ is nonconstant on $C$. In particular, for each $i=1,\cdots,m$, setting $z_i=\tr_{a_i}\pi(p_i)\in\C$ the set
$$C(A)\cap \{\rho\in C(\C):\tr\rho(a_i)=z_i\}$$
is finite. It therefore remains to consider
$$C(A)\setminus\bigcup_{i=1}^m\{\rho\in C(\C):\tr\rho(a_i)=z_i\}.$$
But by the boundedness of each $\pi^*\tr_{a_i}$ near $p_i$ and the discreteness of $A$, the above set lies in a compact subset of $C(\C)$ (under the analytic topology). Again by the discreteness of $A$, this shows that $C(A)$ is finite, as desired.

To prove the last assertion, note that the desired curves $a_i$ can in principle be found effectively by simply enumerating and going through the list of all isotopy classes of essential curves on $\Sigma$, with Corollary \ref{corthm1} guaranteeing the termination of this procedure. Once these functions are found, the above proof leads to the desired effective bounds on the sizes of points in $C(A)$.
\end{proof}

Theorem \ref{theorem3} yields a broad generalization and strengthening of the following result of Long and Reid \cite{lr}. Let $M$ be a finite volume hyperbolic three-manifold with a single cusp. By the work of Thurston, its character variety $X(M)$ has an irreducible component $C(M)$, containing the faithful discrete representation of $\pi_1 M$, which is an algebraic curve. As in Section \ref{sect:1.2}, let $O_d$ denote the ring of integers of the imaginary quadratic field $\Q(\sqrt{-d})$ for each squarefree integer $d>0$.

\begin{ntheorem}[Long and Reid \cite{lr}] For $M$ as above, the set $\bigcup_{d>0} C(M)(O_d)$ is finite.
\end{ntheorem}

The inclusion of the cuspidal torus $T\to M$ induces a morphism from $C(M)$ into the so-called Cayley cubic algebraic surface $X(T)$, which is isomorphic to $X_2(\Sigma)$ where $\Sigma$ is a surface of type $(1,1)$. The crucial ingredient in the proof of the above \cite[Lemma 3.4]{lr} is that the image of $C(M)$ in $X_2(\Sigma)$ is a nonintegrable curve. Given this, Theorem \ref{theorem3} readily recovers the above theorem on the integral points of $C(M)$, bypassing a somewhat involved arithmetical argument in \cite{lr}.

\subsection{Application of Baker's theory} \label{sect:4.2}
In this interlude, we demonstrate a result on certain lattice points lying on algebraic curves in algebraic tori. We shall obtain it as a straightforward consequence of Baker's theory on linear forms in logarithms. Let $\G_m$ denote the multiplicative group, and let $d\geq1$ be an integer. Let $\overline\Q\subset\C$ be the field of algebraic numbers. Fix a sequence of real algebraic numbers
$$(z_1,\cdots,z_d)\in\G_m^d(\overline\Q\cap\R)$$
with each $z_i\neq\pm1$, and let $\Gamma\leq \G_m^d(\overline\Q\cap\R)$ be the subgroup consisting of elements of the form $(z_1^{l_1},\cdots,z_d^{l_d})$ with $l_i\in\Z$. Let $C\subset\G_m^d$ be an irreducible algebraic curve defined over $\overline\Q$, and fix $p\in\G_m^d(\overline\Q)$. Our goal is to prove the following.

\begin{proposition}
\label{baker}
One of the following holds:
\begin{enumerate}
	\item $C(\overline\Q)\cap(\Gamma\cdot p)$ is finite, or
	\item $C$ is a translate of an algebraic subtorus $T\leq\G_m^d$ defined over $\overline\Q$ and invariant under some nontrivial $z\in\Gamma$.
\end{enumerate}
Moreover, $C(\overline\Q)\cap(\Gamma\cdot p)$ can be effectively determined.
\end{proposition}

\begin{proof}
We may assume $p=(1,\cdots,1)$ up to translation. We proceed by induction on $d\geq1$. The claim is obvious if $d=1$, so we may assume that $d\geq2$. Assuming the result in the case $d=2$, we shall first show below how the general case follows.

Suppose that $C(\overline\Q)\cap\Gamma$ is infinite. Up to rearranging the factors of $\G_m^d$, we may assume that the projection morphism $\pi:\G_m^{d}\to\G_m^{d-1}$ onto the first $d-1$ factors is nonconstant along $C$, since otherwise our claim is clear. Now, the Zariski closure $C'$ of $\pi(C)$ in $\G_m^{d-1}$ contains infinitely many points lying in the group
$$\Gamma'=\{(z_1^{l_1},\cdots,z_{d-1}^{l_{d-1}}):l_i\in\Z\}\leq\G_m^{d-1}(\overline\Q\cap\R).$$
By inductive hypothesis, $C'$ is a translate of an algebraic subtorus $T'\leq\G_m^{d-1}$, and $C'$ is preserved by some nontrivial $z'\in\Gamma'$. Up to translation of $C$ within $\G_m^d$, we may assume that $C'=T'$. Let us also denote by $z'\in\G_m(\overline\Q\cap\R)$ the element corresponding to $z'$ under the identification $T'\simeq\G_m$. Applying the case $d=2$ of the proposition to the immersion $C\hookrightarrow T'\times\G_m\subset\G_m^d$, we see that $C$ is a translate of an algebraic subtorus in $T'\times\G_m$ which is invariant under some nontrivial element of $\langle z'\rangle\times\langle z_d\rangle\leq\Gamma$. This gives the desired result.

It remains to prove the proposition in the case $d=2$. We begin by making a number of simplifying remarks. First, it suffices to prove the proposition with $\Gamma$ replaced by the monoid
$$\{(z_1^{l_1},z_2^{l_2}):l_i\in\Z_{\geq0}\},$$
since $\Gamma$ is a union of finitely many monoids of the above type (obtained by replacing $(z_1,z_2)$ with $(z_1^{\pm1},z_2^{\pm1})$). Similarly, it suffices to treat the case where $|z_1|,|z_2|>1$, since we reduce to this case by applying inversions to factors of $\G_m^2$ appropriately.

Let $f\in\overline\Q[X,Y]$ be an irreducible polynomial defining the Zariski closure of $C$ in $\A^2$ under the obvious embedding $\G_m^2\to\A^2$. In what follows, we shall write $(x,y)=(z_1,z_2)$ for convenience, so that in particular $|x|,|y|>1$ by our assumption above. Assuming that the set
$$S=\{(m,n)\in\Z_{\geq0}^2:f(x^m,y^n)=0\}$$
is infinite, we shall show that $C$ is a translate of algebraic torus and is preserved by some nontrivial $z\in\Gamma$. Let us write
$$f(X,Y)=\sum_{i=1}^r a_iX^{d_i}Y^{e_i}$$
with nonzero $a_i\in\overline\Q$ and $d_i,e_i\in\Z_{\geq0}$, such that $(d_i,e_i)\neq(d_j,e_j)$ whenever $i\neq j$. Note that the number $r$ of terms in the above sum is at least $2$ since $x,y\neq0$. Upon relabeling the terms and passing to an infinite subset of $S$, we may assume that
$$(r-1)|a_2x^{d_2m}y^{e_2n}|\geq|a_1x^{d_1m}y^{e_1n}|\geq\cdots\geq|a_rx^{d_rm}y^{e_rn}|$$
for all $(m,n)\in S$. Let $q=\gcd(d_1-d_2,e_1-e_2)>0$, and define $d,e\geq 0$ to be coprime integers such that $|d_1-d_2|=qd$ and $|e_1-e_2|=qe$. The above inequalities imply that there is a constant $K\geq1$ with
\begin{align*}
\tag{$*$}K^{-1}\leq\left|\frac{(x^m)^d}{(y^n)^e}\right|\leq K
\end{align*}
for all $(m,n)\in S$. Assigning weights $e$ and $d$ to the variables $X$ and $Y$ respectively, let us write
$$f(X,Y)=g(X,Y)+h(X,Y)$$
where $g$ is the top degree weighted homogeneous part of $f$, and the remainder $h$ is a polynomial of lower degree. We remark that $g$ must include the term $a_1X^{d_1}Y^{e_1}$. Indeed, for any $i=1,\cdots,r$ we must have
\begin{align*}
|a_i|K^{-d_i/d}|y^{n}|^{(d_ie+de_i)/d}&\leq |a_i(x^m)^{d_i}(y^n)^{e_i}|\\
&\leq|a_1(x^m)^{d_1}(y^n)^{e_1}|\leq|a_1|K^{d_1/d}|y^n|^{(d_1e+de_1)/d}.
\end{align*}
Thus, if $d_1e+de_1<d_ie+de_i$, then $|y|^n$ must be bounded for every $(m,n)\in S$, and by a similar argument $|x|^m$ must be bounded for every $(m,n)\in S$, which is a contradiction. We also have
\begin{align*}
\deg(a_2X^{d_2}Y^{e_2})&=d_2e+de_2=(d_2-d_1)e+d_1e+d(e_2-e_1)+de_1\\
&=(qde-qde)+(d_1e+de_1)=\deg(a_1X^{d_1}Y^{e_1})
\end{align*}
by definition, so $g$ also includes the term $a_2X^{d_2}Y^{e_2}$. In fact, we see that $a_i X^{d_i}Y^{e_i}$ is a term in $g$ if and only if $(d_i,e_i)=(d_1+q'd,e_1-q'e)$ for some $q'\in\Z$. This shows that we can factorize $g(X,Y)$ as
$$g(X,Y)=cX^{\lambda}Y^{\mu}\prod_j(\alpha_j X^d-Y^e)^{\nu_j}$$
for some nonzero $c\in\overline\Q$ and pairwise distinct nonzero $\alpha_j\in\overline\Q$. Up to relabeling and passing to an infinite subset of $S$, there exists some $\kappa>0$ with
\begin{align*}
\left|\alpha_1\frac{(x^m)^d}{(y^n)^e}-1\right|^{\nu_1}&=\frac{|g(x^m,y^n)|}{|c(x^{m})^{\lambda}(y^n)^{\mu}\prod_{j\neq1}(\alpha_j x^{md}-y^{ne})^{\nu_j}|}\frac{1}{|y^{ne}|^{\nu_1}}\\
&=\frac{1}{|c(x^{m}/y^n)^{\lambda}|\prod_{j\neq1}|\alpha_j (x^{md}/y^{ne})-1|^{\nu_j}}\frac{|h(x^m,y^n)|}{|y^{ne}|^{(\lambda +\mu)/e+\sum_j\nu_j}}\\
&\ll \frac{1}{|y^{ne}|^{\kappa}}\ll\frac{1}{\max\{|x^d|^m,|y^e|^n\}^{\kappa}}.
\end{align*}
for every $(m,n)\in S$. Here, we have used the inequality $(*)$. We remark that we must have $\alpha_1\in\overline\Q\cap\R$. In particular, there exists $M>1$ such that
$$|\alpha_1|\frac{|x^d|^m}{|y^e|^n}=1+O\left(\frac{1}{M^{\max\{m,n\}}}\right).$$
for every $(m,n)\in S$.
Taking logarithms and using the fact that $|\log(1+\xi)|\leq 2|\xi|$ for every $|\xi|\leq1/2$, upon passing to a suitable infinite subset of $S$ we have
$$|\log|\alpha_1|+m\log |x^d|-n\log |y^e||\ll\frac{1}{M^{\max\{m,n\}}}.$$
By Baker's theorem \cite[Theorem 3.1]{baker}, if $\log|\alpha_1|$, $\log|x^d|$, and $\log|y^e|$ were $\Q$-linearly independent, then we would have
$$|\log|\alpha_1|+m\log |x^d|-n\log |y^e||\gg\max\{m,n\}^{-D}$$
for some effective constant $D>0$. Therefore, what we have obtained shows that $x$ and $y$ must be multiplicatively dependent. Let $v>1$ be the unique real algebraic number such that $(|x|,|y|)=(v^a,v^b)$ for positive coprime integers $a,b\in\Z$. We may rewrite $(*)$ as
$$K^{-1}\leq v^{adm-ben}\leq K.$$
This shows that $adm-ben$ takes at most finitely many values, and so there is some $t\in\Z$ such that, up to passing to an infinite subset of $S$, we have $adm-ben=t$ for all $(m,n)\in S$. Fix $(m_0,n_0)\in S$. We then have
$$ad(m-m_0)=be(n-n_0)$$
for every $(m,n)\in S$. Since $ad$ and $be$ are coprime, for each $(m,n)\in S$ we have $N=N_{m,n}\in\Z$ with $(m,n)=(m_0,n_0)+N\cdot(be,ad)$. Consider now the morphism
$$\Phi:\A^1\to\A^2$$
given by $\Phi(T)=(x^{m_0}T^e,y^{n_0}T^d)$. Note that $\Phi$ restricts to a morphism $\G_m\to\G_m^2$ so that the image of $\Phi$ is a translation of an algebraic torus. Furthermore, given $(m,n)\in S$ and $N=N_{m,n}$ as above, we have
\begin{align*}
\Phi(v^{Nab})&=(x^{m_0}v^{Nabe},y^{n_0}v^{Nabd})=(x^{m_0}(v^{abe})^N,y^{m_0}(v^{abd})^N)\\
&=(x^{m_0}x^{beN},y^{n_0}y^{adN})=(x^m,y^n).
\end{align*}
This shows that the Zariski closure of $C$ has infinitely many points in common with the image of $\Phi$. This implies that $\Phi(\G_m)=C$. This proves Proposition \ref{baker}.
\end{proof}

\subsection{Integrable curves} \label{sect:4.3}
Fix a pants decomposition $P\subset\Sigma$. Let $\Gamma_P$ be the free abelian subgroup of rank $3g+n-3$ in the mapping class group $\Gamma(\Sigma)$ generated by Dehn twists along the curves in $P$. The action of $\Gamma_P$ on $X_k$ preserves the fiber $X_{k,t}^P$ for each $t\in X(P,\C)$. For fixed $t=(t_1,\cdots,t_{3g+n-3})$, let us fix a sequence $(z_1,\cdots,z_{3g+n-3})\in(\C^{\times})^{3g+n-3}$ of complex numbers such that
$$z_i+z_i^{-1}=t_i\quad\text{for all}\quad i=1,\cdots,3g+n-3.$$
Let $\Gamma_z$ be the subgroup of $\G_m^{3g+n-3}(\C)$ generated by translations by $z_i$ (in the multiplicative sense) in the $i$th coordinate. Our discussion in Section \ref{sect:2.2.3} leads us to the following result.

\begin{proposition}
\label{perfprop}
If $X_{k,t}^P$ is perfect, then there is a morphism
$$F:\G_m^{3g+n-3}\to X_{k,t}^P$$
of schemes (defined over $\overline\Q$ if $k$ and $t$ are algebraic) satisfying the following:
\begin{enumerate}
	\item At the level of complex points, $F$ is surjective with finite fibers,
	\item the action of $\Gamma_z$ on $\G_m^{3g+n-3}$ lifts the $\Gamma_P$-action on $X_{k,t}^P$.
\end{enumerate}
\end{proposition}

\begin{proof}
We shall describe the map induced by $F$ on the complex points. It will be conceptually clear from our construction, even if laborious to show, that the map is induced from a morphism of schemes, and is moreover defined over $\overline\Q$ provided that $k$ and $t$ are algebraic.

To construct $F$, we first fix an $\SL_2(\C)$-local system $\rho_0$ on $\Sigma$ whose class lies in the fiber $X_{k,t}^P$. Let us write $P=a_1\sqcup\cdots\sqcup a_{3g+n-3}$ with each $a_i$ a curve, on which we fix a base point $x_i\in a_i$. Let $\alpha_i$ be a choice of a simple loop based at $x_i$ parametrizing $a_i$. We fix a trivalization of the fiber of $\rho$ above each $x_i$ so that the monodromy along $\alpha_i$ is given by a diagonal matrix of the form
\begin{align*}
\tag{$*$}\begin{bmatrix}z_i & 0 \\ 0 & z_i^{-1}\end{bmatrix}\in\SL_2(\C).
\end{align*}
This is possible since $t_i\neq\pm2$ by our hypothesis that $X_{k,t}^P$ is perfect. Setting aside the case $(g,n,k)=(1,1,2)$ which is elementary, this hypothesis also implies that the restriction of $\rho_0$ to each connected component of $\Sigma|P$ is irreducible, and in fact determines the isomorphism type of such restriction for every local system $\rho$ whose class lies in $X_{k,t}^P$.

For each $i=1,\cdots,3g+n-3$, let us denote by $a_i'$ and $a_i''$ the two boundary curves on $\Sigma|P$ corresponding to $a_i$, and let $(x_i',\alpha_i')$ and $(x_i'',\alpha_i'')$ be the corresponding lifts of $(x_i,\alpha_i)$, respectively. We shall assume that we have chosen the labelings so that the interior of $\Sigma|a$ lies ot the left as one travels alons $\alpha_i'$. The above observation shows that any local system with class in $X_{k,t}^P$ is determined by the isomorphisms of local systems $\rho_0|_{a_i'}$ and $\rho_0|_{a_i''}$ compatible with the gluing of $a_i'$ and $a_i''$ in $\Sigma$. For instance, $\rho_0$ itself is the local system determined by the identity isomorphisms
$$\id:\rho_0|_{a_i'}\simeq\rho_0|_{a_i}\simeq\rho_0|_{a_i''}.$$
More generally, an isomorphism of local systems $\rho_0|_{a_i'}\simeq\rho_0|_{a_i''}$ is specified by an element in the centralizer of the matrix $(*)$ above (namely, the group of diagonal matrices of determinant one). Thus, we have a map
$$F:\G_m^{3g+n-3}(\C)\to X_{k,t}^P(\C)$$
sending $(v_1,\cdots,v_{3g+n-3})\in\G_m^{3g+n-3}(\C)$ to the class of the local system determined by the isomorphisms
$$\begin{bmatrix}v_i & 0\\ 0 & v_i^{-1}\end{bmatrix}:\rho_0|_{a_i'}\simeq\rho_0|_{a_i''}.$$
The fact that this is surjective follows from the above discussion. To see that $F$ has finite fibers, note that if we glue the local systems $\rho_0|_{\Sigma'}$ on the components of $\Sigma'$ of $\Sigma|P$ along the curves $a_1,\cdots,a_{3g+n-3}$ in a fixed order, at each stage the resulting local system is uniquely determined, except possibly when $a_i$ is a separating curve in which case there is a double ambiguity.

Finally, the compatibility of the action of $\Gamma_z$ with the action of $\Gamma_P$ under $F$ follows from the our description of lifts of Dehn twists in Section \ref{sect:2.2.3}.
\end{proof}

\begin{corollary}
\label{perfectcor}
Let $X_{k,t}^P$ be a perfect fiber. Then
\begin{enumerate}
	\item $|\Gamma_P\backslash X_{k,t}^P(A)|<\infty$ for any closed discrete $A\subset\R$, and
	\item if no coordinate of $t\in X(P,\C)\simeq\A^{3g+n-3}$ lies in $[-2,2]$, then
	$$|\Gamma_P\backslash X_{k,t}^P(A)|<\infty$$
	for any closed discrete $A\subset\C$.
\end{enumerate}
\end{corollary}

\begin{proof}
Let us write $t=(t_1,\cdots,t_{3g+n-3})$, and let us first suppose $t_i\notin[-2,2]$ for each $i=1,\cdots,3g+n-3$. This shows in particular that each $z_i$ (defined above so that $z_i+z_i^{-1}=t_i$) has absolute value different from $1$. In particular, every point in $\G_m^{3g+n-3}(\C)$ is $\Gamma_z$-equivalent to a point in a region $K\subset\G_m^{3g+n-3}(\C)$ which is compact with respect to the Euclidean topology. Under the map
$$F:\G_m^{3g+n-3}(\C)\to X_{k,t}^P(\C)$$
constructed in Proposition \ref{perfprop}, the image of $K$ in $X_{k,t}^P(\C)$ is compact and hence has finite intersection with any closed discrete $A\subset\C$. The equivariance property of $F$ proved in Proposition \ref{perfprop} then implies our claim, when $t_i\notin[-2,2]$ for each $i=1,\cdots,3g+n-3$.

Let us now assume that $A$ is a closed discrete subset of $\R$. Note in particular that the coordinates $t_i$ may be assumed to lie in $A$ and hence are real, since otherwise $X_{k,t}^P(A)$ is empty. Now, let us consider the natural morphism
$$X_{k,t}^P(\Sigma)\to\prod_{i=1}^{3g+n-3}X_{k_i,t_i}^{a_i}(\Sigma_i)$$
where each $\Sigma_i$ is the surface of type $(0,4)$ or $(1,1)$ obtained by gluing together the two boundary curves on $\Sigma|P$ corresponding to $a_i$, and the boundary traces $k_i$ are appropriately determined from $k$, $P$, $t$, and $\Sigma_i$. The fact that the map $F$ constructed in Proposition \ref{perfprop} has finite fibers implies that the above morphism also has finite fibers (at the level of complex points). The claim to be proved thus reduces to the case where $\Sigma$ is of type $(0,4)$ or $(1,1)$. But this is obvious by elementary geometric considerations (see also the proof of Theorem 1.4 in \cite{whang2}).
\end{proof}

\begin{proposition}
\label{integrable}
Let $C$ be a geometrically irreducible algebraic curve over $\Z$ lying in a perfect fiber $X_{k,t}^P$. Then $C(\Z)$ can be effectively determined, and
\begin{enumerate}
	\item $C(\Z)$ is finite, or
	\item $C(\Z)$ is finitely generated under some nontrivial $\gamma\in\Gamma_P$ preserving $C$.
\end{enumerate}
If moreover $C$ is not fixed pointwise by any nontrivial $\gamma\in\Gamma_P$, the same result holds with $C(\Z)$ replaced by the set of all imaginary quadratic integral points on $C$.
\end{proposition}

\begin{proof}
Note that $t\in X(P,\Z)\simeq\Z^{3g+n-3}$. We shall first consider the case where no coordinate of $t$ lies in $[-2,2]$. By Corollary \ref{perfectcor}, the set $\bigcup_{d>0}X_{k,t}^P(O_d)$ consists of finitely many $\Gamma_P$-orbits. Thus, it suffices to consider the intersection of $C(\C)$ with the orbit of a single point $\rho\in\bigcup_{d>0} X_{k,t}^P(O_d)$.

Let $F:\G_m^{3g+n-3}\to X_{k,t}^P$ be a $\Gamma_P$-equivariant morphism as constructed in the proof of Proposition \ref{perfprop}. Let $z=(z_1,\cdots,z_{3g+n-3})\in\G_m^{3g+n-3}$ be as defined earlier in this subsection. By our hypothesis on $t$, each $z_i$ is a real quadratic integer. Let $C'=F^{-1}(C)$, and choose a point $p\in F^{-1}(\rho)$. Applying Proposition \ref{baker} to the curve $C'\subset\G_m^{3g+n-3}$ and projecting the result down to $X_{k,t}^P$, we obtain the result.

Let us write $P=a_1\sqcup\cdots\sqcup a_{3g+n-3}$, and let us denote by $t_i\in\Z$ the component of $t$ corresponding to $a_i$. Based on the above argument, it remains to consider the case where $t_i\in[-2,2]$ for some $i$. Since we must have $t_i\in\{0,\pm1\}$, we see that the Dehn twist $\tau_{a_i}$ acts on the fiber $X_{k,t}^P$ with finite order (as seen from the discussion in Section \ref{sect:2.3}), so we need only to consider $C(\Z)$. Let $\Sigma_i$ be the surface of type $(0,4)$ or $(1,1)$ obtained by gluing together the two boundary curves on $\Sigma|P$ corresponding to $a_i$, and consider the composition of morphisms
$$C\to X_{k,t}^P(\Sigma)\to X_{k_i,t_i}^{a_i}(\Sigma_i)$$
where $k_i\in X(\del\Sigma_i,\C)$ is appropriately determined from $k$, $t$, and $\Sigma_i$. Note that the set of real points of $X_{k_i,t_i}^{a_i}(\Sigma_i)$ defines a ellipse in an appropriate coordinate plane, as seen from our discussion in Section \ref{sect:2.2.3}. In particular, $X_{k_i,t_i}^{a_i}(\Sigma_i,\Z)$ is finite, and if the above composition is nonconstant then we find that $C(\Z)$ is finite, as desired. It thus remains to consider the case where the above composition is constant. This implies that the morphism
$$C\to X_{k,t}^P(\Sigma)\to X_{k',t'}^{P'}(\Sigma|a_i)$$
(where $k'$, $P'$, and $t'$ are appropriately determined from $k$, $P$, $t$, and $a_i$)
must be nonconstant, as seen from the consideration of the morphism $F:\G_m^{3g+n-3}\to X_{k,t}^P$ as constructed in the proof of Proposition \ref{perfprop}. Thus, we may apply induction and a straightforward modification of the previous paragraph to conclude the result.
\end{proof}

\subsection{Proofs of Theorem \ref{theorem4} and Corollary \ref{corollary5}} \label{sect:4.4}
We obtain Theorem \ref{theorem4} by combining Theorem \ref{theorem3} for nonintegrable curves with Proposition \ref{integrable} for integrable curves in $X_k$. Finally, Corollary \ref{corollary5} follows easily from Theorem \ref{theorem4} (and indeed from Theorem \ref{theorem3}) and our understanding of the fibers $X_{k,t}^P$ from Section \ref{sect:3.1}.

\end{document}